\documentclass[11pt]{article}
\usepackage{vmargin}
\usepackage{times}

\usepackage[T1]{fontenc}
\usepackage{epsfig}
\usepackage{graphicx}
\usepackage{amssymb}
\usepackage{amsmath}
\usepackage[utf8]{inputenc}
\usepackage{amsfonts}
\usepackage{amsthm}
\usepackage{paralist}
\usepackage{stmaryrd}
\usepackage[toc,page]{appendix}

\usepackage[colorlinks=true]{hyperref}
\usepackage[dvipsnames]{xcolor}

\theoremstyle{plain}
\newtheorem{corollary}{Corollary}[section]
\newtheorem{theorem}[corollary]{Theorem}
\newtheorem{lemma}[corollary]{Lemma}
\newtheorem{proposition}[corollary]{Proposition}
\newtheorem*{theorem*}{Theorem}
\newtheorem*{lemma*}{Lemma}

\theoremstyle{definition}
\newtheorem{algorithm}{Algorithm}

\newtheorem*{definition*}{Definition}
\newtheorem{Hypo}{Assumption}

\theoremstyle{remark}

\newtheorem*{remark}{Remark}

\newcommand{\R}{\mathbb{R}}

\newcommand{\proba}{\mathbb{P}}
\newcommand{\N}{\mathbb{N}}
\newcommand{\E}{\mathbb{E}}

\newcommand{\T}{\mathcal{T}}
\newcommand{\GT}{\mathfrak{G}}

\newcommand{\M}{\mathfrak{p}}

\newcommand{\D}{\mathcal{D}}
\newcommand{\V}{\mathcal{V}}
\newcommand{\Ee}{\mathcal{E}}

\newcommand{\Dn}{{\D_n}}
\newcommand{\Pn}{{\P_n}}
\newcommand{\Tn}{{\Theta_n}}
\newcommand{\K}{\mathbb{K}}
\newcommand{\B}{\mathcal{B}}

\newcommand{\G}{\mathcal{G}}

\newcommand{\1}{\mathbf{1}}
\renewcommand{\P}{\mathcal{P}}

\newcommand{\AAA}{\mathfrak{E}}

\newcommand{\n}{s}

\newcommand{\p}{{\mathfrak p}}
\newcommand{\e}{\varepsilon}
\newcommand{\s}{\sigma}

\newcommand{\ply}{\Rightarrow}
\newcommand{\OmegaD}{\Omega_{\D}}
\newcommand{\OmegaP}{\Omega_{\P}}
\newcommand{\OmegaT}{\Omega_{\Theta}}
\newcommand{\OmegaM}{\Omega_{\M}}

\DeclareMathOperator*{\limit}{\longrightarrow}

\DeclareMathOperator*{\SBB}{SB}
\DeclareMathOperator*{\CB}{CB}
\DeclareMathOperator*{\cyc}{cyc}

\newcommand*{\SB}{\text{SB}}
\newcommand*{\AB}{\text{AB}}
\newcommand*{\SBb}{\text{\textbf{yz}}}

\newcommand*{\GH}{\text{GH}}

\newcommand*{\GP}{\text{GP}}
\newcommand*{\W}{\text{W}}
\newcommand*{\WGP}{\text{WGP}}
\newcommand*{\WGH}{\text{WGH}}

\DeclareMathOperator*{\Id}{Id}

\DeclareMathOperator*{\CMmm}{CM}
\newcommand*{\CM}{{\CMmm}}
\DeclareMathOperator*{\MGmm}{MG}
\newcommand*{\MG}{{\MGmm}}
\DeclareMathOperator*{\cov}{cov}
\DeclareMathOperator*{\egale}{=}
\begin{document}
%\tableofcontents
%\newpage
\title{Limit of connected multigraph with fixed degree sequence} %\thanks{Not final version. Some proof reading remains to be done. And some parts are subject to high change, notably 8.1.}
\author{Arthur Blanc-Renaudie\thanks{LPSM, Sorbonne Universit\'e, France Email: arthur.blanc-renaudie@sorbonne-universite.fr.}}
\date{\today}
\maketitle
 \begin{abstract} Motivated by the scaling limits of the connected components of the configuration model, we study uniform connected multigraphs with fixed degree sequence $\D$ and with surplus $k$. We call those random graphs $(\D,k)$-graphs. We prove that, for every $k\in \N$, under natural conditions of convergence of the degree sequence, ($\D,k)$-graphs converge toward either $(\P,k)$-graphs or $(\Theta,k)$-ICRG (inhomogeneous continuum random graphs). We prove similar results for $(\P,k)$-graphs and $(\Theta,k)$-ICRG, which have applications to multiplicative graphs.  Our approach relies on two algorithms, the cycle-breaking algorithm, and the stick-breaking construction of $\D$-tree that we introduced in a recent paper. From those algorithms we deduce a biased construction of $(\D,k)$-graph, and we prove our results by studying this bias.
 %Hence our results have applications to several other models of random trees.
\end{abstract}
\section{Introduction} %
The present work is a continuation of our previous paper \cite{Uniform}, where we introduced a stick-breaking construction for $\D$-trees (uniform tree with fixed degree sequence $\D$) to prove that, under natural conditions, $\D$-trees converge, in a GP and a GHP sens, toward either $\P$-trees or ICRT. Here, we derive from \cite{Uniform} similar limits for graph versions of those trees, which have applications to multiplicative graphs and to the configuration model.
\subsection{Motivations} 
Computer scientists have introduced multiplicative graphs \cite{HistoMultA,HistoMultC,HistoMultD} and the configuration model \cite{HistoConfigA,HistoConfigB} as natural generalizations of the Erd\H os--R\'enyi model. They are studied for 2 main reasons: first many tools introduced for the Erd\H{o}s--R\'enyi model can also be used to study those graphs, then those models seems closer to real life network thanks to the "inhomogeneity in their degree distribution" (see e.g. Newman \cite{Newman}). For those reasons, they are currently great models to study the evolution of random networks. 

A natural question for any model of evolution is to study their potential phase transitions. It appears that those graphs have an intriguing phase transition where a giant component gets born.  We refer the reader to \cite{Dhara} Chapter 1 and references therein for an elaborate discussion of the nature of this transition, and an  overview of the literature it generated. 

From the point of view of precise asymptotics, a main goal is to study the geometry of the connected components of those graphs in the critical regime. To this end,  Addario-Berry, Broutin, and Goldschmidt \cite{ABG} have developed a general approach in the case of the Erd\H os--R\'enyi model. This approach is divided in two main steps:
\begin{compactitem}
\item[(a)]  First one encodes the random graphs into stochastic processes, and study those processes to deduce several limits for relevant quantities of the largest connected components such as the size, surplus, degrees. This has been noticed in the ground-breaking work of Aldous \cite{Aldous_exc_ER}.
\item[(b)] Then, one use those convergences to reduce the problem to a study of a single connected component conditioned on those quantities.
\end{compactitem}

This approach has been further developed for multiplicative graphs and the configuration model in many different regimes. We refer the reader to \cite{ABG,MST2,HomogeneousCase2} for the homogeneous case, \cite{StableConfig, HeavyConfig, StableMarchal} for the power law case, and \cite{P-graph-2,P-graph-1} for a unified approach for multiplicative graphs. 
In this paper, we focus on solving (b), under what we believe to be the weakest assumptions. So we reduce the study of the largest connected components to solving (a), which tends to be simpler. 

Moreover, we give a universal point of view on those models which can be summarized into the three following points: we describe multiplicative graphs as degenerate configuration model, we extend the unified point of view of Broutin, Duquesne, and Minmin \cite{P-graph-2,P-graph-1} to the configuration model, and we remove the omnipresent randomness assumption in the degree sequence.
\subsection{Overview of the proof} \label{1.2}
Fix $k\in \N$. Fix $\{V_i\}_{i\in \N}$ a  set of vertices. We say that a multigraph $G$ have degree sequence $\D=(d_1,\dots, d_\n)$ if $G$ has vertices $(V_1,\dots, V_\n)$ and for every $1\leq i \leq \n$, $V_i$ has degree $d_i+1$. (This shift of $+1$ will be convenient to simplify many expressions, and to be coherent with \cite{Uniform}.) The surplus of a connected multigraph $(V,E)$ is $|E|-|V|+1$, and is, informally, the number of edges that one needs to delete to transform a multigraph into a tree. A $(\D,k)$-graph is a uniform connected multigraph with degree sequence $\D$ and surplus $k$.

Our goal is to study the connected components of the configuration model conditioned on having degree sequence $\D$ and surplus $k$, which are close from $(\D,k)$-graphs (see Lemma \ref{Connections}).
To this end, we rely on two algorithms: the stick-breaking construction of $\D$-trees of \cite{Uniform}, along with the cycle-breaking algorithm introduced by Addario-Berry, Broutin, Goldschmidt, and Miermont \cite{MST} which we invert to construct $(\D,k)$-graph by adding $k$ edges to a biased $\D$-tree. %Conveniently those two algorithms goes so well together that they reduce our problem to a single computation. 
%We explain below why.
%Our proof strongly relies on two algorithms: the stick-breaking construction of $\D$-tree that we introduced in \cite{Uniform}, along with the cycle-breaking algorithm introduced in \cite{MST} which we invert to construct $(\D,k)$-graph by adding $k$ edges to a biased $\D$-tree. Conveniently those two algorithms goes so well  together that they reduce our problem to a computation. We explain below why.

We use the cycle breaking algorithm in the following form. Take a connected multigraph with surplus $k$, repeat $k$ times: choose an edge uniformly among all the edges that can be removed without disconnecting the graph, then cut this edge in the middle. 
By doing so, we add $2k$ named leaves $(\star_i)_{1\leq i \leq 2k}$, and keep the degrees of $(V_i)_{i\in \N}$.  Note that to invert this algorithm we can intuitively repair the broken edges by gluing the different pairs in $(\star_i)_{1\leq i \leq 2k}$. Note that however this algorithm is not a bijection, since for each multigraph there are many corresponding trees. To bypass this, we bias each tree by the probability that they were obtained by their corresponding multigraph. This way, we construct a $(\D,k)$-graph from a biased $\D$-tree with $k$ additional edges.  %and then reconstruct the multigraph.

Thus, to study the geometry of a $(\D,k)$-graph, it is enough to study jointly the geometry of a $\D$-tree, the positions of $(\star_i)_{1\leq i \leq 2k}$, and the previous bias which is a function of $(d(\star_i,\star_j))_{1\leq i,j \leq 2k}$. Therefore, it is enough to study precisely the distance matrix between specific vertices of a $\D$-tree. If the bias was a continuous function of this matrix, then our main results would directly follow  from \cite{Uniform} since the GP convergence of $\D$-trees implies the convergence of this matrix. However, some extra care is needed since the bias diverges when $(\star_i)_{1\leq i \leq 2k}$ are close. 

Therefore, we need to prove that $(\star_i)_{1\leq i \leq 2k}$ cannot be too close. More precisely we show, using the structure of $\D$-trees and of the bias, that it is enough to lower bound $(d(\star_0,\star_i))_{1\leq i \leq k}$ where $\star_0$ is a root leaf.  We then use our construction of $\D$-trees, also introduced independently by Addario-Berry, Donderwinkel, Maazoun, and Martin in \cite{Steal1},  to lower bound those distances using the $k$ first repetitions in a random tuple. %Some particular care is then paid to vertices of degree 2 since the  is slightly imprecise.

%To this end, we first prove that w 
%
%we use the construction of \cite{Uniform}, to prove that the probability that $\{\star_i\}_{1\leq i \leq a}$ for $a\in \N$ are close from a root leaf $\{\star_0\}$ is small by upper bounding the $a$ first repetitions in a random tuple. Then using, the particular structure of $\D$-trees and of the bias, we upper bound the bias.

%The construction of \cite{Uniform} is suited for such an estimate. Indeed, it can be seen in three different ways: as a modification of the Aldous--Broder algorithm introduced in \cite{AaldousBroder,AldousBbroder}, as a recursive construction of the subtrees spanned by specific leaves, and as a stick-breaking construction.
%By combining those point of views, we upper bound the distance between $(\star_i)_{1\leq i \leq 2k}$ by upper bounding the $2k$ first repetitions in a random tuple.

Finally, since the bias is a function of the subtree spanned by $(\star_i)_{1\leq i \leq 2k}$, it is also a function of the first branches of the stick-breaking construction. This allows us to consider the limit of the bias, to directly construct  the limits of $(\D,k)$-graphs by biasing the $\P$-trees and ICRT, introduced by Aldous, Camarri and Pitman \cite{IntroICRT1,IntroICRT2}, and then by gluing the $k$ first pair of leaves. 
\paragraph{Plan of the paper:} In Section \ref{TOPOsection} we introduce the topologies that we are using in this paper. In Section \ref{2}, we construct $\D$-trees, $\P$-trees, and $\Theta$-ICRT. In Section \ref{ModGraph}, we construct $(\D,k)$-graphs, $(\P,k)$-graphs, and $(\Theta,k)$-ICRG. We state our main results in Section \ref{MAINsection}. We study the bias in section \ref{Bias}. We deduce our main results in Section \ref{May the proof section be removed?}. Finally we discuss in Section \ref{ALTEsection} some connections between $(\D,k)$-graph, $(\P,k)$-graphs, the configuration model and multiplicative graphs. %This last point does not directly support this paper but we need to introduce it for \cite{ExcursionStick}. 

\paragraph{Notations:} Throughout the paper, similar variables for, $\D$-trees, $(\D,k)$-graphs, $\P$-trees, $(\P,k)$-graphs, $\Theta$-ICRT, $(\Theta,k)$-ICRG share similar notations. To avoid any ambiguity, the models that we are using and their parameters are indicated by superscripts $\D$, $(\D,k)$, $\P$ ,$(\P,k)$, $\Theta$, $(\Theta,k)$. We often drop those superscripts when the context is clear.

\paragraph{Acknowledgment} Thanks are due to Nicolas Broutin for many advices on the configuration model and on multiplicative graphs.
%We use the letter $d$ for both degrees and distances. It is always clear which one it refers to.
%\end{notation}
%
\section{Notions of convergence} \label{TOPOsection}
%The reader may wish to simply skim this section on a first reading, referring back to it as needed. 
\subsection{Gromov--Prokhorov (GP) topology} \label{GPdef}
A measured metric space is a triple $(X,d,\mu)$ such that $(X,d)$ is a Polish space and $\mu$ is a Borel probability measure on $X$. Two such spaces $(X,d,\mu)$, $(X',d',\mu')$ are called isometry-equivalent if there exists an isometry $f:X\to X'$ such that if $f_\star \mu$ is the image of $\mu$ by $f$ then $f_\star \mu=\mu'$. Let $\mathbb{K}_{\GP}$ be the set of isometry-equivalent classes of measured metric space. Given a measured metric space $(X,d,\mu)$, we write $[X,d,\mu]$ for the isometry-equivalence class of $(X,d,\mu)$ and frequently use the notation $X$ for either $(X,d,\mu)$ or $[X,d,\mu]$.

%If $(X,d,\mu), (X',d',\mu')\in \mathbb{K}_{\GH}$, let $M(X,X')$ be the set of finite non-negative Borel measures on $X\times X'$. We will denote by $p$, $p'$ the natural projection from $X\times X'$ to $X$ and $X'$. The discrepancy of $\pi\in M(X,X')$ is the quantity
%\[ D(\pi,\mu,\mu')= \| \pi-p_{\star} \mu \| +\| \pi-p'_{\star} \mu' \| \]
%where $\|v\|$ is the total variation of the signed measure $v$.

We now recall the definition of the Prokhorov's distance. Consider a metric space $(X,d)$. For every $A\subset X$ and $\e>0$ let $A^\e:= \{x\in X, d(x,A)<\e\}$. Then given two (Borel) probability measures $\mu$, $\nu$ on $X$, the Prokhorov distance between $\mu$ and $\nu$ is defined by 
\[d_P(\mu, \nu):= \inf\{\text{ $\e>0$: $\mu\{A\}\leq \nu \{A^\e\}$ and $\nu\{A\}\leq  \mu\{A^\e\}$, for all Borel set $A\subset X$} \}.\]

The  Gromov--Prokhorov (GP) distance is an extension of the Prokhorov's distance: For every $(X,d,\mu),(X',d',\mu')\in \K_{\GP}$ the Gromov--Prokhorov distance between $X$ and $X'$ is defined by
\[ d_{\GP}((X,d,\mu),(X',d',\mu')):=\inf_{S,\phi,\phi'} d_P(\phi_\star \mu, \phi'_\star\mu'),\]
where the infimum is taken over all metric spaces $S$ and isometric embeddings $\phi :X\to S$, $\phi' :X'\to S$. $d_{\GP}$ is indeed a distance on $\K_{\GP}$ and $(\K_{\GP},d_{\GP})$ is a Polish space (see e.g. \cite{GHP}).

We use another convenient characterization of the GP topology: For every measured metric space  $(X,d^X,\mu^X)$ let $(x_i^X)_{i\in \N}$ be a sequence of i.i.d. random variables of common distribution $\mu^X$ and let $M^X:=(d^X(x_i^X,x_j^X))_{(i,j)\in \N^2}$. We prove the following result in \cite{Uniform} (see also \cite{EquivGP}),

\begin{lemma} \label{equivGP2} Let $(X^n)_{n\in \N}\in \K_{\GP}^\N$ and let $X\in \K_{\GP}$. Let $(y^X_i)_{i\in \N}$ be a sequence of random variables on $X$ and let $N^X:= (d^X(y_i^X,y_j^X))_{(i,j)\in \N^2}$. If
\[ M^{X_n} \limit^{(d)} N^X \quad \text{and} \quad \frac{1}{n}\sum_{i=1}^n \delta_{y^X_i} \limit^{(d)} \mu^X, \]
then $X^n \limit^{\GP} X$.% and thus $M^X$ and $N^X$ have the same distribution.
\end{lemma}
\subsection{Gromov--Hausdorff (GH) topology} \label{GH}
Let $\K_{\GH}$ be the set of isometry-equivalent classes of compact metric space. For every metric space $(X,d)$, we write $[X,d]$ for the isometry-equivalent class of $(X,d)$, and frequently use the notation $X$ for either $(X,d)$ or $[X,d]$. 

 For every metric space $(X,d)$, the Hausdorff distance between $A,B\subset X$ is given by
\[d_H(A,B):= \inf\{\e>0, A\subset B^\e, B\subset A^\e \}. \]
The Gromov--Hausdorff distance between $(X,d)$,$(X',d')\in \K_{\GH}$ is given by 
\[ d_{\GH}((X,d),(X',d')):=\inf_{S,\phi,\phi'} \left (d_H(\phi(X), \phi'(X')) \right ),\]
where the infimum is taken over all metric spaces $S$ and isometric embeddings $\phi :X\to S$, $\phi' :X'\to S$. $d_{\GH}$ is indeed a distance on $\K_{\GH}$ and $(\K_{\GH},d_{\GH})$ is a Polish space (see e.g.  \cite{GHP}).

\subsection{Pointed Gromov--Hausdorff ($\GH^n$) topology} \label{PointedGH} Let $n\in \N$. Let $(X,d,(x_1,\dots, x_n))$ and $(X',d',(x'_1,\dots, x'_n))$ be metric spaces, each equiped with an ordered sequence of $n$ distinguished points (we call such spaces $n$-pointed metric spaces). We say that these two $n$-pointed metric spaces are isometric if there exists an isometry $\phi$ from $(X,d)$ on $(X',d')$ such that for every $1\leq i \leq n$, $\phi(x_i)=x'_i$. 
 
Let $\K^n_{\GH}$ be the set of isometry-equivalent classes of compact metric space. As before, we write  $[X,d,(x_1,x_2,\dots, x_n)]$ for the isometry-equivalent class of $(X,d,(x_1,\dots, x_n))$, and denote either by $X$ when there is little chance of ambiguity.
 
The $n$-pointed Gromov--Hausdorff distance between $X,X'\in \K^n_{\GH}$ is given by 
\[ d_{\GH}^n((X,d,(x_1,\dots, x_n)),(X',d',(x'_1,\dots, x'_n))):=\inf_{S,\phi,\phi'} \left (d_H(\phi(X), \phi'(X')) \right ),\]
where the infimum is taken over all metric spaces $S$ and isometric embeddings $\phi :X\to S$, $\phi' :X'\to S$ such that for every $1\leq i \leq n$, $\phi(x_i)=\phi'(x'_i)$. $d^n_{\GH}$ is indeed a distance on $\K^n_{\GH}$ and $(\K^n_{\GH},d^n_{\GH})$ is a Polish space (see \cite{MST} Section 2.1).
\subsection{Extension to pseudo metric spaces}
Note that the previous topologies naturally extends to pseudo metric spaces. Indeed, one may say that a pseudo metric space $(X,d)$ is isometry-equivalent to the metric space given by quotienting $X$ by the equivalent relation $d(a,b)=0$ (see Burago, Burago, Ivanov \cite{Glue} for details.) It is then enough to extend the equivalent classes to pseudo metric spaces. %This will be convenient later to glue some leaves of ICRT without changing their labels.
%Since, the distance actually does not depends on the measure, we usually write for every metric spaces $(X,d),(X',d')$, $d_{\GH}((X,d),(X,d')):=d_{\GH}((X,d,\{0\}),(X,d',\{0\}))$.
%\begin{remark} Note that we could have defined $d_{\GH}$ directly on the set of equivalent classes of compact metric space. We work on measure metric space for convenience issue. \end{remark} 
%We add the measure to consider our random variables on the same space for convenience issue.
\section{Constructions of $\D$-trees, $\P$-trees and $\Theta$-ICRT} \label{2}
%The following section is organized as follow. In Section \ref{2.2}, we introduce our new stick-breaking construction of $\D$-trees. In section \ref{2.3} and \ref{2.4}, we construct $\P$-trees and $\Theta$-ICRT, and explain why they appears as the limit of $\D$-trees.
%\paragraph{Plan of the section} We present in Section \ref{2.1} a generic stick breaking construction for $\R$-tree and explain an equivalent point of view for discrete trees. Then we introduce $\D$-trees, $\P$-trees and $\Theta$-ICRT with their respective constructions in Section \ref{2.2}, \ref{2.3} and \ref{2.4}.
\subsection{$\D$-trees} \label{2.2} 
%In this section, we recall the stick-breaking construction of $\D$-trees that we introduced in \cite{Uniform}. 
Recall that a sequence $(d_1,\dots, d_{\n})$ is a degree sequence of a tree if and only if $\sum_{i=1}^\n d_i=\n-2$, and by convention $d_1\geq d_2\dots \geq d_{\n}$. Let $\OmegaD$ be the set of such sequences. %Recall that if $\D\in \OmegaD$, then a $\D$-tree is a random rooted tree uniform among all rooted tree with degree sequence $\D$.

For convenience issue, we want to label our leaves on a set $\{\star_i\}_{i\in \N}$ disjoint from $\{V_i\}_{i\in \N}$. So let us slightly change our definition of $\D$-trees. Note that a tree with degree sequence $\D$ must have $N^\D+2:=\sum_{i=1}^\n \1_{d_i=0}$ leaves. We say that a tree $T$ is a $\D$-tree if it is uniform among all tree with vertices $\{V_i\}_{i:d_i>0}\cup \{\star_i\}_{0\leq i \leq N+1}$ and such that for every $i$ with $d_i>0$, $\deg(V_i)=d_i+1$.
%Since we will want to choose the label of the leaves, let us extend our definition of $\D$-trees. Note that a tree with degree sequence $\D$ must have $N+2:=\sum_{i=1}^\n \1_{d_i=0}$ leaves. We say that a tree $T$ is a $\D$-tree with leaves $(L_i)_{0\leq i \leq N+1}$ if and only if it is uniform among all tree with vertices $\{V_i\}_{i:d_i>0}\cup \{L_i\}_{0\leq i \leq N+1}$ and such that for every $i$ with $d_i>0$, $\deg(V_i)=d_i+1$. We often omit the label of the leaves when they are clear or when they do not matter.

We now recall the construction of $\D$-trees of \cite{Uniform}. For simplicity,  for every graph $G=(V,E)$ and edge $e=\{v_1,v_2\}$, $G\cup e$ denotes the graph $(V\cup \{v_1,v_2\},E\cup \{e\})$. 

 %Also, for every $\D\in \OmegaD$ we let $L^\D_0$, $L^\D_1$, \dots be the leaves (that is the vertices $V_{a_1},V_{a_2}\dots$ with $a_1\leq a_2\leq \dots$ and $d_{a_1}=d_{a_2}=\dots =0$).

%This paper features two new stick breaking constructions for $\D$-trees. The first constructs recursively $(T(L_1,\dots, L_i))_{1\leq i \leq N+1}$ where $L_1,\dots, L_{N+1}$ are the leaves (vertices with $0$ child). The second one, which is presented in Section \ref{6.2}, is more general and construct recursively $(T(W_1,\dots, W_i))_{1\leq i \leq \n}$ where $W_1,\dots, W_\n$ is an arbitrary permutation of $\V^\D:=\{V_1$,\dots, $V_{\n}\}$. 
\begin{algorithm}[Algorithm 7 from \cite{Uniform}] \label{D-tree} \emph{Stick-breaking construction of a $\D$-tree $T^\D$ (see Figure \ref{explore1}).}
\begin{compactitem}
\item[-] Let $A^\D=(A^\D_i)_{1\leq i \leq \n-1}$ be a uniform $\D$-tuple (tuple such that $\forall i\in \N$, $V_i$ appears $d_i$ times).%(tuple such that for every $i\in \N$, $V_i$ appears $d_i$ times.)
\item[-] %Define the sequence of tree $\{T_i\}_{1\leq i \leq \n}$ by induction as follow: 
Let $T^\D_1:=(\{\star_0,A_1\},\{\{\star_0,A_1\}\})$ then for every $2\leq i \leq \n$ let 
\[ T^\D_i:=\begin{cases} T_{i-1}\cup \{A_{i-1},A_{i}\} & \text{if } A_{i}\notin T_{i-1}.\\T_{i-1}\cup \{A_{i-1},\star_{\inf\{k, \star_k\notin T_{i-1}\}}\}& \text{if } A_{i} \in T_{i-1} \text{ or } i=\n.
\end{cases} \]
%\begin{compactitem}
%\item[-] if $A_{i}\notin T_{i-1}$, let $T_{i}:= T_{i-1}\cup  \{A_{i-1},A_{i}\}$%and let $\kappa_{i}:=\kappa_{i-1}$,
%\item[-] if $A_{i} \in T_{i-1}$ or $i=\n$, let $T_{i}:= T_{i-1}\cup  \{A_{i-1},L_{\inf\{k, L_k\notin T_{i-1}\}}\}$ %and let $\kappa_{i}:=\kappa_{i-1}+1$,
%\end{compactitem} 
%where by convention $A_\n\in T_{\n-1}$.
\item[-] Let $T^\D=T_\n$.
\end{compactitem}
\end{algorithm}
%In the whole paper for $\D\in \Omega_\D$, $A^\D=(A_1^\D,  \dots,A_{\n^\D}^\D)$ denotes a uniform $\D$-tuple, and $T^\D$ denotes the tree constructed from $A^\D$ and Algorithm \ref{D-tree}.
\begin{figure}[!h] \label{explore1}
\centering
\includegraphics[scale=0.5]{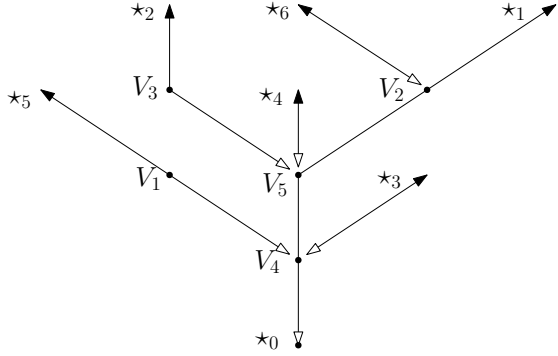}
\caption{Stick breaking construction of a $\D$-tree with $\D=(1,2,1,3,3,0,0,\dots)$ and $(A^\D_i)_{1\leq i \leq \n-1}=(V_4,V_5,V_2 ,V_5,V_3,V_4,V_5,V_4,V_1,V_2)$. The exploration starts at $\star_0$ then follows the white-black arrow toward $\star_1$, then jumps at $\star_5$ to follow the path toward $\star_2$ and so on\dots
} %(V_4,V_5,V_2 ,V_5,V_3,V_4,V_5,V_4,V_1,V_2,V_1) %(V_4,V_5,V_2 ,{\color{blue}V_5},V_3,{\color{blue}V_4},{\color{blue}V_5},{\color{blue}V_4},V_1,{\color{blue}V_2},{\color{blue}V_1}) 
\label{explore1}
\end{figure}%
%\begin{remark} %$\bullet$ 
%The labels of the leaves $L_0, L_1 \dots, $ are irrelevant to our study since they do not change the geometry of the tree. To match our formal definition of $\D$-trees, they can be any permutation of the vertices $V_i$ such that $d_i^\D=0$. Also, it will be convenient to use other labels. \\
%$\bullet$ Conveniently if $\D=(d_1,\dots, d_\n)\in \OmegaD$, we let $(d_1,\dots, d_\n, 0,\dots, 0)$-tree denote a $\D$-tree (with an arbitrary number of $0$). There is no issue involved since the number of leaves is given by the other degrees.
%\end{remark} 
%Note that there is $N=\sum_{i=1}^\n (d_i-1)\1_{d_i\geq 1}$ indexes $Y_1$, \dots $Y_N$ such that $A_i\in \{A_1,\dots, A_{i-1}\}$.
%Also, for every $1\leq i \leq N$, let $Z_i$ denotes the smallest integer $z$ such that $A_z=A_{Y_i}$. We call $\{Y_i\}_{1\leq i \leq N}$ the cuts, and $\{Z_i\}_{1\leq i \leq N}$ the glue points. 
\subsection{$\P$-trees} \label{2.3}
Let $\{V_{\infty,i}\}_{i\in \N}$ be a set of vertices disjoint with $\{V_i\}_{i\in \N}$ and $\{\star_i\}_{i\geq 0}$.
 Let $\OmegaP$ be the set of sequence $(p_i)_{i\in \N\cup \{\infty\}}$ in $\R^+$ such that $\sum_{i=1}^\infty p_i+p_{\infty}=1$, $p_1>0$ and $p_1\geq p_2\geq \dots$. For every $\P\in \OmegaP$, the $\P$-tree is the random tree constructed as follow:
\begin{algorithm} \label{P-tree} \emph{Definition of the $\P$-tree for $\P\in \OmegaP$.}
\begin{compactitem}
\item[-] Let $(A^\P_i)_{i\in \N}$ be a family of i.i.d. random variables such that for all $i\in \N$, $\proba(A^\P_1=V_i)=p_i$.
\item[-] For every $i\in \N$, let $B^\P_i=A_i$ if $A_i\in \N$, and let $B^\P_i=V_{\infty,i}$ otherwise.
\item[-] %Define the sequence $\{T_i\}_{i\in \N}$ by induction as follow: 
Let $T^\P_1:=(\{\star_0,B_1\},\{\{\star_0,B_1\}\})$ then for every $i\geq 2$ let
\begin{equation*} T^\P_i:=\begin{cases} T_{i-1}\cup \{B_{i-1},B_{i}\} & \text{if } B_{i}\notin T_{i-1}. \\T_{i-1} \cup \{B_{i-1},\star_{\inf\{k, \star_k\notin T_{i-1}\}}\}& \text{if } B_{i} \in T_{i-1}. % \{B_{i-1},L_{\inf\{k, L_k\notin T_{i-1}\}}\}
\end{cases}  \end{equation*} %\label{1002}
%\begin{compactitem}
%\item[-] if $B_{i}\notin T_{i-1}$, let $T_{i}:= T_{i-1}\cup  \{B_{i-1},B_{i}\}$
%\item[-] if $B_{i} \in T_{i-1}$, let $T_{i}:= T_{i-1}\cup  \{B_{i-1},L_{\inf\{k, L_k\notin T_{i-1}\}}\}$
%\end{compactitem} 
\item[-] Let $T^\P:=\bigcup_{n\in \N}T_n$.
\end{compactitem}
\end{algorithm}
%\begin{remarks} (a) We keep the leaves $\{L_i\}_{i\in \N}$ of $T^\P$ in \eqref{1002} to avoid measurability issue. \\
%(b) Note that the infinite vertex $V_\infty$ is split into many others. Indeed this vertex "fills the gap" produced by vertices with small degree. The introduction of $V_{\infty}$ allows us to consider a slightly more general definition than the one introduced in \cite{IntroICRT2} which requires $\sum_{i=1}^{+\infty} p_i=1$. 
%For those reasons $T^\P$ is no longer a tree with vertices in $\{V_i\}_{i\in \N}$. 
%\end{remarks}
\begin{remark} Usually, the leaves $\{\star_i\}_{i\in \N}$ are omitted in the formal definition of $\P$-trees. We consider them to clarify the intuition that they are degenerate $\D$-trees with an infinite number of leaves.
\end{remark}

\subsection{ICRT} \label{2.4}
\begin{figure}[!h]  \label{SB}
\centering
\includegraphics[scale=0.6]{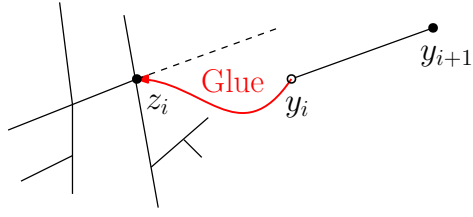}
\caption{A typical step of the stick-breaking construction: the "gluing" of $(y_i,y_{i+1}]$ at $z_i$. }
\label{SB}
\end{figure}
First let us introduce a generic stick breaking construction.  It takes for input two sequences in $\R^+$ called cuts ${\textbf y}=(y_i)_{i\in \N}$ and glue points ${\textbf z}=(z_i)_{i\in \N}$, which satisfy
\begin{equation} \forall i<j,\ \ y_i<y_j \qquad ; \qquad y_i\limit \infty \qquad ; \qquad \forall i\in \N,\ \ z_i\leq y_i, \label{2609} \end{equation}
and creates a $\R$-tree (loopless geodesic metric space) by recursively "gluing" segment $(y_i,y_{i+1}]$ on position $z_i$, %(see Figure \ref{SB})%(a loopless geodesic space see Le Gall \cite{Legall} for an extensive treatment)
  or rigorously, by constructing  a consistent sequence of distances $(d_n)_{n\in \N}$ on $([0,y_n])_{n\in \N}$.
\begin{algorithm} \label{Alg1} \emph{Generic stick-breaking construction of $\R$-tree.}
\begin{compactitem}
\item[--] Let $d_0$ be the trivial metric  on $[0,0]$.
\item[--] For each $i\geq 0$ define the metric $d_{i+1}$ on $[0, y_{i+1}]$ such that for each $x\leq y$: 
\[ d_{i+1}(x,y):=
\begin{cases} 
d_{i}(x,y) & \text{if } x,y\in [0, y_i] \\
d_{i}(x,z_i)+|y-y_i| & \text{if } x \in [0, y_i], \, y \in (y_i, y_{i+1}] \\
|x-y| &  \text{if } x,y\in (y_i, y_{i+1}],
\end{cases} \]
where by convention $y_0:=0$ and $z_0:=0$.
\item[--] Let $d$ be the unique metric on $\R^+$ which agrees with $d_i$ on $[0, y_i]$ for each $i\in \N$.
\item[--] Let $\SBB({\textbf y},{\textbf z})$ be the completion of $(\R^+,d)$.
\end{compactitem}
\end{algorithm}

Now, let $\OmegaT$ be the space of sequences $(\theta_i)_{i\in \{0\}\cup \N}$ in $\R^+$ such that $\sum_{i=0}^\infty \theta_i^2=1$ and such that $\theta_1\geq \theta_2\geq \dots$. For every $\Theta\in \OmegaT$, the $\Theta$-ICRT is the random $\R$-tree constructed as follow:
\begin{algorithm} \label{ICRT} \emph{Construction of $\Theta$-ICRT (from \cite{ICRT1})}
\begin{compactitem}
\item[-] Let $(X_i)_{i\in \N}$ be a family of independent exponential random variables of parameter $(\theta_i)_{i\in \N}$.
\item[-] Let $\mu$ be the measure on $\R^+$ defined by $\mu=\theta_0^2 dx+\sum_{i=1}^{\infty} \delta_{X_i} \theta_i$. 
%\item[-] For each $l\in\R^+$ let $\mu_l$ be the restriction of $\mu$ to $[0,l]$.
\item[-] Let $(Y_i,Z_i)_{i\in \N}$ be a Poisson point process  on $\{(y,z)\in \R^{+2}: y\geq z \}$ of intensity $dy\times d\mu$.
\item[-] Let $\textbf{Y}:=(Y_i)_{i\in \N}$ and let $\textbf{Z}:=(Z_i)_{i\in \N}$. Let $(Y_0,Z_0):=(0,0)$.
%\item[-] Let $(Z_i)_{i\geq 1}$ be a family of independent random variables with respective laws $(\frac{\mu_{Y_i}}{\mu[0,Y_i]})_{i\geq 1}$.
\item[-] The $\Theta$-ICRT is defined as $(T,d)=\SBB(\textbf Y,\textbf Z)$. (see Algorithm \ref{Alg1})
\end{compactitem}
\end{algorithm}
%\begin{remark} It is still possible to define $\Theta$-ICRT whenever $\theta_0=0$ or $\sum_{i=1}^\infty \theta_i<\infty$. We exclude those cases because when $\D$-trees converge to those ICRT, no GP convergence is possible. To be more precise, we saw in \cite{Uniform} that in some condensation cases the  distance matrix between random vertices converge weakly toward the distance matrix of $\{Y_i^\Theta\}_{i\in \N}$, but that $\sum_{i=1}^\infty \delta_{Y^\Theta_i}$ does not converge toward a probability measure. Thus there is no probability measure on those ICRT suited for the GP convergence.  However, as in \cite{Uniform} we prove the convergence of the distance matrix of $(\D,k)$-graph without using this assumption.
%\end{remark} 
\section{Constructions of $(\D,k)$-graphs, $(\P,k)$-graphs and $(\Theta,k)$-ICRG} \label{ModGraph}
%For the rest of the paper $k\in \N$ denotes a fixed integer.
\subsection{Generic gluing and cycle-breaking of discrete multigraphs (see Figure \ref{explore2})} \label{DefGraph} %
In the entire section, $G=(V,E)$ denotes a multigraph. 
Let $\cyc(G)$ the set of all edges $e\in E$ such that $G\backslash \{e\}:=(V,E\backslash\{e\})$ is connected. (For multiples edges the operation $\backslash$ only remove one edge at a time.)  Let $\square(G):=\#\cyc(G)$. 

For every leaves $L_1\neq L_2\in G$, we define the operation of gluing $L_1$ and $L_2$ in $G$ as follow: For every leaf $L\in G$, let the father of $L$ be the only vertex $F\in G$ such that $(F,L)\in G$. Let $F_1$, $F_2$ be the father of $L_1,L_2$. The multigraph obtained by gluing $L_1$ and $L_2$ in $G$ is %the multigraph 
\[ \G_{L_1,L_2}(G):=(V/\{L_1,L_2\}, E\backslash \{ \{F_1,L_1\},\{F_2,L_2\} \}\cup \{ \{F_1,F_2\}\}), \]
and intuitively corresponds to the graph obtained by fusing $\{F_1,L_1\}$ and $\{F_2,L_2\}$.

Similarly, for every leaves $L_1\neq L_2 \neq \dots, L_{2k-1} \neq L_{2k}$, the multigraph obtained by gluing $L_1$ and $L_2$,\dots, $L_{2k-1}$ and $L_{2k}$ in $G$ is
\[ \G_{(L_i)_{1\leq i \leq 2k}}(G)=\G_{(L_1,L_2),(L_3,L_4)\dots, (L_{2k-1},L_{2k})}(G):= \G_{L_{1},L_{2}}\circ \G_{L_3,L_4} \circ \dots \circ \G_{L_{2k-1},L_{2k}}(G). \] Note that this multigraph does not depend on the order in which we glue the different leaves.
\begin{figure}[!h] \label{explore2}
\centering
\includegraphics[scale=0.55]{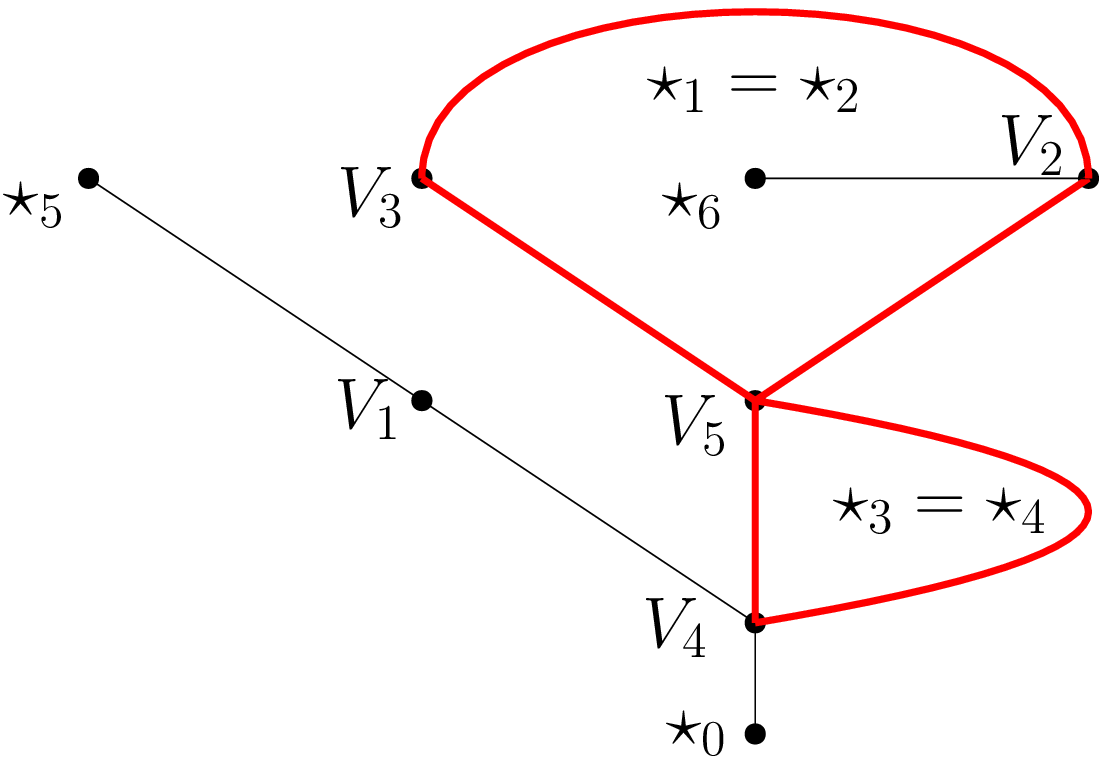}
\caption{Gluing leaves of the tree  $T$ from Figure \ref{explore1} to form a graph $G$ with surplus 2. $\cyc(G)$ is red. $\square(G)=5$.  $G=\G_{(\star_1,\star_2,\star_3,\star_4)}(T)$. Also, $\proba(\CB(G)=T)=\frac{2!}{2^2\square(G\backslash\{V_4,V_5\})\square(G)}=\frac{2}{2^2*3*5}$.} %(V_4,V_5,V_2 ,V_5,V_3,V_4,V_5,V_4,V_1,V_2,V_1) %(V_4,V_5,V_2 ,{\color{blue}V_5},V_3,{\color{blue}V_4},{\color{blue}V_5},{\color{blue}V_4},V_1,{\color{blue}V_2},{\color{blue}V_1}) 
\label{explore2}
\end{figure}%

Now recall Section \ref{1.2}. Let us give a formal definition of the cycle-breaking algorithm:
\begin{algorithm} \label{CBalg} Cycle-breaking of a multigraph $G=(V,E)$ with $V\subset (V_i)_{i\in \N}$ and surplus $k$.
\begin{compactitem} 
\item[-] For $1\leq i \leq k$, let $e_i=(W_{2i+1},W_{2i+2})$ be a uniform oriented edge in $\cyc(G\backslash \{e_j\}_{1\leq j<i})$.
\item[-] Let $\CB(G):= (V\cup\{\star_i\}_{1 \leq i \leq k},(E\backslash \{e_i\}_{1\leq i \leq k})\cup \{\{W_i, \star_{2k+1-i}\}\}_{1\leq i \leq 2k})$.
\end{compactitem}
\end{algorithm}
To simplify our notations for every multigraph $G=(V,E)$ and $v,w\in V$, we write $\#_{v,w}(G)$ for the number of edges $\{v,w\}$ in $G$. Also, let $\circ(G):=\prod_{v\in V}2^{\#_{v,v}(G)}\prod_{v,w\in V}\#_{v,w}(G)!$.
\begin{lemma} \label{CBapp} For every connected multigraph $G$ with $V\subset \{V_i\}_{i\in \N}$ and surplus $k$, we have:
\begin{compactitem}
\item[(a)] $\CB(G)$ is almost surely a tree with vertices $V\cup \{\star_i\}_{1\leq i \leq 2k}$.
\item[(b)] For every $v\in V$, $\deg_{\CB(G)}(v)=\deg_G(v)$. For every $1\leq i \leq 2k$, $\star_i$ is a leaf in $\CB(G)$.
\item[(c)] Almost surely, $\G_{(\star_1, \star_2),\dots, (\star_{2k-1},\star_{2k})}(\CB(G))=G$.%for every $1\leq i \leq k$,
% \[ \G_{(\star_{1},\star_{2}),\dots, (\star_{2i-1},\star_{2i})}(\CB(G))=G\backslash \{e_j\}_{1\leq j\leq i}. \] Notably a.s. 
\item[(d)] For every tree $T$ satisfying (a) (b),
\begin{equation} \proba(\CB(G)=T)= \frac{\circ(G)}{2^k\prod_{i=1}^{k} \square(G\backslash \{e_j\}_{1\leq j<i})}. \label{1910} \end{equation}
\end{compactitem}
\end{lemma}
\begin{proof} (a) and (b) follows from a quick enumeration. (c) is easy to prove from the definition of $\G$.
 (d) follows from an induction. Indeed, the right hand side of \eqref{1910} is just the product over each steps of the probability that $(W_{2i+1},W_{2i+2})$ satisfies   $\{ W_{2i+1},\star_{2k-2i}\},\{W_{2i+2},\star_{2k-2i-1}\}\in T$.
\end{proof}
%Informally, Lemma \ref{CBapp} tells us that Algorithm \ref{CBalg} is a randomized "bijection" from multigraphs with degree sequence $\D$ and surplus $k$ toward trees with degree sequence $\D$. This "bijection" have a natural inverse $\G_{(\star_1, \star_2),\dots, (\star_{2k-1},\star_{2k})}$. This inverse is key in the following constructions.
 
 \subsection{$(\D,k)$-graph} \label{DkDef}
Note that $(d_1,\dots, d_{\n})$ is a degree sequence of a connected multigraph with surplus $k$ if and only if $\sum_{i=1}^\n d_i=\n+k-2$, and by convention $d_1\geq d_2\dots \geq d_{\n}$. Note that by adding $2k$ numbers $0$, this holds if and only if $(d_1,\dots, d_{\n},0,\dots, 0)\in \OmegaD$. %For this reason, we let $\OmegaDk$ denote the subsets of sequence of $\D$ that have at least $2k$ 0.

For convenience issue, let us slightly extend our definition of $(\D,k)$-graph. For $\D\in \OmegaD$ with $N^\D\geq 2k$ we say that $G$ is a $(\D,k)$-graph if it is uniform among all multigraph with vertices $\{V_i\}_{i:d_i>0}\cup \{\star_i\}_{i\in\{0\}\cup  \{2k+1,\dots, N^\D+1\}}$ and such that for every $i$ with $d_i>0$, $\deg(V_i)=d_i+1$. The following result follows from Lemma \ref{CBapp} and constructs a $(\D,k)$-graph from a biased $\D$-tree.
%Again, we want to label our leaves on $\{\star_i\}_{i\in \N}$. Note that a $\Dk$ graph must have $N^\Dk+2:=\sum_{i=1}^{\n}\1_{d_i=2}$ leaves 
%o for every $\D\in \OmegaDk$, let $N^\Dk:=\sum_{i=1}^{\n}$ we say that a tree $T$ is a $\D$-tree iff it is uniform among all tree with vertices $\{V_i\}_{i:d_i>0}\cup \{\star_i\}_{0\leq i \leq N+1}$ and such that for every $i$ with $d_i>0$, $\deg(V_i)=d_i+1$.

 \begin{lemma} \label{constructDK}Let $T^{\D,k}$ be a random tree. Assume that for every tree $T$ such that: T have vertices $\{V_i\}_{1\leq i\leq \n}\cup \{\star_i\}_{1\leq i \leq 2k}$, for every $1\leq i \leq \n$ $\deg_T(V_i)=d_i+1$, and $\{\star_i\}_{1\leq i \leq 2k}$ are leaves of $T$,
 \begin{equation} \proba(T^{\D,k}=T)\alpha  \frac{\circ(\G_{(\star_i)_{1\leq i \leq 2k}}(T))}{\prod_{i=1}^{k} \square(\G_{(\star_{1},\star_{2}),\dots, (\star_{2i-1},\star_{2i})}(T))}, \label{19102}\end{equation}
where $\alpha$ stands for proportional. %and where $\Nn_k(T)$ denotes the number of couples of leaves in $\{\{\star_{2i+1},\star_{2i+2}\}\}_{1\leq i \leq k}$ that have the same father in $T$. 
Then $\G_{(\star_i)_{1\leq i \leq 2k}}(T^{\D,k})$ is a $(\D,k)$-graph.
 \end{lemma} 
 To simplify our notations, we write for every $i\in \N$, $\square_i(\cdot):=\square(\G_{(\star_{1},\star_{2}),\dots, (\star_{2i-1},\star_{2i})}(\cdot))$ and $\square_{\square,k}(\cdot):= \circ(\G_{(\star_i)_{1\leq i \leq 2k}}(T))/\prod_{i=1}^{k} \square_i(\cdot)$. So that the right hand side of \eqref{19102} is $\square_{\square,k}(T)$. %Moreover, note that for every $n\in \N$, $\square_n(T)$ is a measurable function of $(d_T(\star_i, \star_j))_{1\leq i \leq j \leq 2k}$, where $d_T$ stands for the graph distance in $T$. (To do so, one can reconstruct the structure of the subtree spanned by $\{\star_i\}_{1\leq i \leq 2k}$ from this matrix and then recover $\square_k(T)$).
% \begin{remark} The term $\circ(\G_{(\star_i)_{1\leq i \leq 2k}}(T))$ does not change our proofs since it is a bounded function, and is $1$ with high probability.% and one may have replaced it by any bounded function. In particular, if one consider the connected component of the configuration model this term disappear. 
% \end{remark}
\subsection{$(\P,k)$-graph}  \label{PkDef}
Since $\P$-trees appear at the limit of $\D$-trees, it is natural to adapt Lemma \ref{constructDK} to construct limits for $(\D,k)$-graphs from $\P$-trees. Thus we informally define the $(\P,k)$-graph as a $\P$-tree biased by \eqref{1910} where we glued $\{\star_{2i-1},\star_{2i}\}_{1\leq i \leq k}$. Below we formally define $(\P,k)$-graph.

Fix $\P\in \OmegaP$. First note that Algorithm \ref{P-tree} can be seen as a function $\AB$ (Aldous--Br\"oder) which takes a tuple $A^\P$ in $\Omega_\AB:=(\{V_i\}_{i\in \N\cup\{\infty\}})^\N$ and send a tree $T^\P$. We equip $\Omega_\AB$ with the weak topology and let $\B_\AB$ be the Borel algebra of this space. Also, we equip $\Omega_\AB$ with the distribution $\proba^\P$ of $(A_i^\P)_{i\in \N}$, and complete the space so that event of measure null for $\proba^\P$ are measurable.

Then note that $\square_{\square,k}(\AB)$ is a measurable function from $\Omega_\AB$ to $\R^+$ since it is locally constant on the subspace of tuple that have at least $2k$ repetitions. Also, note that $\square_{\square,k}(\AB)\leq (k+1)!2^k$. Thus we may define $\proba^{\P,k}$ on $(\Omega^\AB,\B^\AB)$ such that for every Borel space $B\in \B_\AB$,
\[ \proba^{\P,k}(B)=\E[\1_{A^\P\in B}\square_{\square,k}(\AB(A^\P))]/\E[\square_{\square,k}(\AB(A^\P))]. \]

Now let $A^{\P,k}$ be a random variable with distribution $\proba^{\P,k}$. Then let $T^{\P,k}:=\AB(A^{\P,k})$. The $(\P,k)$-graph is the random graph $G^{\P,k}:=\G_{(\star_i)_{1\leq i \leq 2k}}(T^{\D,k})\backslash \{\star_i\}_{i\in \N}$.
\subsection{$(\Theta,k)$-ICRG} \label{4.4} \label{TkDef}
Since $\Theta$-ICRT appear  as the limit of $\D$-trees it is natural to adapt Lemma \ref{constructDK} to construct limits for $(\D,k)$-graphs from $\Theta$-ICRT. Thus we informally define $(\Theta,k)$-ICRG as $\Theta$-ICRT biased by \eqref{1910} where we glued $\{\star_{2i-1},\star_{2i}\}_{1\leq i \leq k}$. Below we formally define $(\Theta,k)$-ICRG. We stay rudimentary and refer to Chapter 3 of \cite{Glue} or to the $\R$-graph theory of \cite{MST} for more details. 

First we formally define the gluing of two points: For every pseudo metric space $(X,d)$ and $x_1,x_2\in X$ let $\GT_{x,y}((X,d))$ be the pseudo metric space $(X,d')$ where for every $a_1,a_2\in X$,
\[ d'(a_1,a_2):=\inf \{d(a_1,a_2)\,; \,d(a_1,x_1)+d(a_2,x_2) \, ;\,  d(a_1,x_2)+d(a_2,x_1)\}. \]
Also for every $k\in \N$ and $x_1, x_2, \dots, x_{2k}\in X$ let 
\[ \GT_{(x_i)_{1\leq i \leq 2k}}((X,d))=\GT_{(x_1,x_2),\dots, (x_{2k-1},x_{2k})}((X,d)):=\GT_{x_1,x_2}\circ \GT_{x_3,x_4}\circ \dots \circ  \GT_{x_{2k-1},x_{2k}}((X,d)). \]
One can check that $\GT_{(x_i)_{1\leq i \leq 2k}}((X,d))$ does not depends on the order in $(x_{2i-1},x_{2i})_{1\leq i \leq k}$.

Recall Section \ref{2.4}. Let $\K_{\SBb}$ be the set of couples of sequences $\textbf y$ and $\textbf z$ satisfying \eqref{2609}.
In Section \ref{2.4} we defined the stick breaking construction as a function $\SB:(\textbf y, \textbf z)\in \K_{\SBb}\to\SB(\textbf y, \textbf z)$. 

For every $n\in \N$ and $(\textbf y,\textbf z)=((y_i)_{i\in \N}, (z_i)_{i\in \N})\in \K_{\SBb}$ let $\cyc_n(\textbf y, \textbf z)$ be the set of $x\in \R$ such that $\GT_{(y_i)_{1\leq i \leq 2n}}(\SB(\textbf y,\textbf z))\backslash\{x\}$ is connected. Note that $\cyc_n(\textbf y, \textbf z)$ is a finite union of interval so is measurable. Let $\square_n(\textbf y, \textbf z)$ be its Lebesgue measure. Note that $\square_n(\textbf y, \textbf z)$ only depends on $\{y_i\}_{1\leq i \leq n},\{z_i\}_{1\leq i \leq n}$, and is a measurable function of $(\{y_i\}_{1\leq i \leq n},\{z_i\}_{1\leq i \leq n} )$ (see Lemma \ref{CHIANTa}). %(To this end, one may explicitely construct $\cyc_k(\textbf y, \textbf z)$ by induction by adding the $k$ first edges, then prove that $\square_k(\textbf y, \textbf z)$ is a non continuous piecewise linear function.)
Let $\square_{\square,k}(\textbf y, \textbf z):=1/\prod_{n=1}^k \square_n(\textbf y, \textbf z)$.

Let $\mathbb{M}$ be the set of all positive locally finite measure on $\R^+$. Let $\K_{\SB}:= \mathbb{M}\times \K_{\SBb}$. We equip $\K_{\SB}$ with the weak topology and let $\B_\SB$ be the Borel algebra of this space.
Let $\Theta\in \OmegaT$. We will prove in Lemma \ref{Followers} that $\E[\square_{\square,k}(\textbf Y^\Theta, \textbf Z^\Theta)]<\infty$. Thus we may define $ \proba^{\Theta,k}$ on $(\K_{\SB},\B_\SB)$ such that for every Borel space $B\in \B_\SB$,
\[ \proba^{\Theta,k}(B)=\E \left [\1_{A^\P\in B}\square_{\square,k} \left (\SB\left ( \textbf Y^\Theta, \textbf Z^\Theta \right ) \right )\right ]/\E[\square_{\square,k}(\textbf Y^\Theta, \textbf Z^\Theta)]. \]

Now let $(\mu^{\Theta,k}, \textbf Y^{\Theta,k} , \textbf Z^{\Theta,k})$ be a random variable with distribution $\proba^{\Theta,k}$. Let $\textbf Y^{\Theta,k}=(Y_i^{\Theta,k})_{i\in \N}$. Then let $(T^{\Theta,k},\bar d^{\Theta,k}):=\SB( \textbf Y^{\Theta,k} , \textbf Z^{\Theta,k})$. The $(\Theta,k)$-ICRG is the random pseudo metric space $(G^{\Theta,k},d^{\Theta,k}):=\GT_{(Y_i^{\Theta,k})_{1\leq i \leq 2k}}(T^{\Theta,k},\bar d^{\Theta,k})$.
%Finally let us define a probability measure on $G^{\Theta,k}$. We proved in \cite{ICRT1} that a.s. $\sum_{i=1}^n \delta_{Y^\Theta_i}$ converges weakly toward a probability measure $\p^{\Theta}$ on $T^\Theta$. Thus a.s. $\sum_{i=1}^n \delta_{Y^{\Theta,k}_i}$ converges weakly toward a probability measure $\p^{\Theta,k}$ on $T^{\Theta,k}$. Since convergence in $T^{\Theta,k}$ imply convergence in $G^{\Theta,k}$, note that it still makes sens to define $\p^{\Theta,k}$ on $G^{\Theta,k}$.

\section{Main results} \label{MAINsection}
 In this section $(\Dn)_{n\in \N}$, $(\Pn)_{n\in \N}$, $(\Tn)_{n\in \N}$ denote fixed sequences in $\OmegaD$, $\OmegaP$, $\OmegaT$ respectively. 
 For every $\D=(d_1,\dots, d_{\n^\D})\in \OmegaD$, let $(\sigma^\D)^2:=\sum_{i=1}^\n d_i (d_i-1)$ then let $\lambda^\D:=\sigma^\D/\n^\D$. Also,
for every $\P=(p_i)_{i\in \N\cup \{\infty\}}$ let $\n^\P:=\max\{i\in \N\cup \{\infty\}: p_i>0\}$ and let $(\sigma^\P)^2=\sum_{i=1}^{\infty} (p_i)^2$. 
We always work under one of the following regimes:

\begin{Hypo}[$\Dn \ply \P$] \label{Hypo1} For all $i\geq 1$, $d_i^\Dn/\n^\Dn\to p^\P_i$ and $\n^\Dn\to \infty$. 
\end{Hypo}
\begin{Hypo}[$\Dn \ply \Theta$]\label{Hypo2}
  \label{Hypo2} For all $i\geq 1$, $d_i^\Dn/\s^\Dn\to \theta^\Theta_i$ and $d_1^\Dn/\n^\Dn\to 0$.
\end{Hypo} 
\begin{Hypo}[$\Pn \ply \Theta$]\label{Hypo2P} For all $i\geq 1$, $p_i^\Pn/\sigma^\Pn\to \theta^\Theta_i$ and $p_1^\Pn\to 0$.
\end{Hypo} 
\begin{Hypo}[$\Tn \ply \Theta$]\label{Hypo2T} For all $i\geq 1$, $\theta_i^{\Tn}\to \theta^\Theta_i$.
\end{Hypo} 
%Note that when $\Dn\ply \P$ or $\Dn\ply \Theta$, $N^\Dn\to \infty$ so $(\D,k)$-graph are well defined.
\paragraph{A few words on $\ply$.} One can put a topology on $\Omega:=\OmegaD\cup \OmegaP\cup \OmegaT$ such that $\ply$ corresponds with the notion of convergence on $\Omega$. This has several advantages (see \cite{Uniform} Section 8.1 for details). 
First $(\Omega,\ply)$ is a Polish space.  Moreover, our results can be seen as continuity results for the function which associate to a set of parameters a metric space. Hence, our results can be used to study graph with random degree distributions. Furthermore $\Omega_\D$ is dense on $\Omega$. So our results on $(\D,k)$-graphs imply the others. 
\subsection{The bias does not diverge}
%Let us introduce some notations.  Then for every $\D\in \OmegaD$, and $x\in \R^+$ let 

As explained previously in the introduction, our approach relies entirely on the stick breaking construction of \cite{Uniform} and on the study of the bias corresponding to the cycle-breaking construction. More precisely given the following result, our main results are applications of \cite{Uniform}.
\begin{proposition} \label{MainMainMain}  %Fix $k\in \N$. %We have the following results:
%\begin{compactitem} 
%\item[(a)] For every $\e>0$, for every $\D\in \OmegaD$ with $N^\D\geq 2k$, $\E [ \1_{d(\star_0,\star_1) \leq \e}/d(\star_0,\star_1)  ] \leq 6\e$. \item[(b)]
%For every $x,y\in \R$ let $h_y(x):=x \1_{x\geq M}$.
%For every $\delta>0$, there exists $M>0$ such that for every $\D\in \OmegaD$ with $N^\D\geq 2k$, writing for $x\in \R$,  % if $T^\D$ is a $\D$-tree with leaves $\{\star_i\}_{0\leq i \leq N^\D+1}$,
For every $x,m\in \R^+$ let $h_m:=x\1_{x\geq m}$. We have,
\[ \lim_{m\to \infty} \max_{\D\in \OmegaD:N^\D\geq 2k }\E\left [h_{m}\left (\frac{\square_{\square,k}(T^{\D})}{(\lambda^\D)^k}\right )\right ]=0. \]
%\end{compactitem}
\end{proposition}
%Furthermore  $\OmegaDk$ is dense, so our results on $(\D,k)$-graphs implies the others. 
\subsection{Gromov--Prokhorov convergence} \label{resultGP}
First let us specify the measures that we consider.
Let $\OmegaM$ the set of measures on $\{V_i\}_{i\geq 1}\cup\{\star_i\}_{i\geq 0}$. We say  that a sequence $(\M_n)_{n\in \N}\in \OmegaM^\N$ converges toward $\M\in \OmegaM$ if $\max_{i\in \N} |\M_n(V_i)- \p(V_i) | \to 0$ and $\max_{i\in \N} |\M_n(\star_i)- \p(\star_i) | \to 0$. In the whole paper, for every $\D\in \OmegaD$, $\M^{\D,k}$ denote a probability measure with support on $\V^{\D,k}:=\{V_i,i:d_i\geq 1\}\cup \{ \star_i, i\in \{0\}\cup \{2k+1,\dots ,N^\D+1\}\}$. Similarly, for every $\P\in \OmegaP$, $\M^\P$ denote a probability measure with support on $\V^\P:=\{V^\P_i\}_{i:p_i>0}$. Also, we sometimes let $0$ denote the null measure. %(otherwise our results does not make any sens). 

Then we recall the probability measure on ICRT of \cite{ICRT1}. To simplify our expressions, we write $\mu^\Theta=\infty$ when either $\theta^\Theta_0>0$ or $\sum_{i=1}^\infty \theta^\Theta_i=\infty$, (since $\mu^\Theta=\infty$ iff a.s. $\mu^\Theta[0,\infty]=\infty$). %(In those case a.s.  )

\begin{definition*}[\cite{ICRT1} Proposition 3.2]  Let $\Theta\in \OmegaT$ be such that $\mu^\Theta=\infty$. Almost surely, as $n\to \infty$, $\frac{1}{n}\sum_{i=1}^n \delta_{Y_i^\Theta}$ converges weakly toward a probability measure $\p^\Theta$ on $\T^\Theta$.
\end{definition*} 

\begin{remark} When $\mu^\Theta<\infty$, $\frac{1}{n}\sum_{i=1}^n \delta_{Y_i^\Theta}$ does not converge. For this reason, although we prove the convergence of the distance matrices, one cannot define a proper measure for the GP convergence.
\end{remark}
Then let us define a probability measure on $G^{\Theta,k}$. It directly follows from \cite{ICRT1} Proposition 3.2, that a.s. $\sum_{i=1}^n \delta_{Y^{\Theta,k}_i}$ converges weakly toward a probability measure  $\p^{\Theta,k}$ on $T^{\Theta,k}$. Since convergence in $T^{\Theta,k}$ imply convergence in $G^{\Theta,k}$, it still makes sense to define $\p^{\Theta,k}$ on $G^{\Theta,k}$.

We now state the main result of this section. In what follows, $d^{\D,k}$ is the graph distance on $G^{\D,k}$ and similarly $d^{\P,k}$ is the graph distance on $G^{\P,k}$. %Also we write $\mu^\Theta=\infty$ when either $\theta^\Theta_0>0$ or $\sum_{i=1}^\infty \theta^\Theta_i=\infty$.
%\pagebreak[2]
\begin{theorem} \label{GP} The following convergences hold weakly for the GP topology
\begin{compactitem} 
%\item[(a)] If $\Dn\ply \P$ and $\M^\Dn\to \M^\P$ then
%$( \V^\Dn,d^\Dn,\M^\Dn ) \to ^{\WGP} (\V^\P, d^\P,\M^\P).$ %\left ( \V^\Dn,\frac{\s^\Dn}{\n^\Dn}d^\Dn,\Mn \right)
%\item[(b)] If $\Dn \ply \Theta$, $\M^\Dn \to 0$, and $\mu^\Theta=\infty$ then
%$ ( \V^\Dn, (\s^\Dn/\n^\Dn)  d^\Dn,\M^\Dn ) \to^{\WGP} (\T^\Theta,d^\Theta,\p^\Theta).  $
%\item[(c)] If $\Pn \ply \Theta$, $\M^\Dn \to 0$ and $\mu^\Theta=\infty$ then
%$ ( \V^\P,\s^\Pn d^\Pn,\M^\Pn ) \to^{\WGP} (\T^\Theta,d^\Theta,\p^\Theta).  $
%\item[(d)] If $\Tn \ply \Theta$, $\M^\Dn \to 0$ and $\mu^\Theta=\infty$ then
%$(\T^\Tn,d^\Tn,\p^\Tn) \to^{\WGP} (\T^\Theta,d^\Theta,\p^\Theta).  $
\item[(a)] If $\Dn\ply \P$ and $\M^{\Dn,k}\to \M^\P$ then
\[ \left (G^{\Dn,k},d^{\Dn,k},\M^{\Dn,k} \right) \limit^{\WGP} (G^{\P,k}, d^{\P,k},\M^\P).\] %\left ( \V^\Dn,\frac{\s^\Dn}{\n^\Dn}d^\Dn,\Mn \right)
\item[(b)] If $\Dn \ply \Theta$, $\M^{\Dn,k} \to 0$, and $\mu^\Theta=\infty$ then
\[ \left ( G^{\Dn,k},\lambda^\Dn d^{\Dn,k},\M^{\Dn,k} \right) \limit^{\WGP} (G^{\Theta,k},d^{\Theta,k},\p^{\Theta,k}).  \]
\item[(c)] If $\Pn \ply \Theta$, $\M^\Pn \to 0$, and $\mu^\Theta=\infty$ then
\[ \left ( G^{\Pn,k},\s^\Pn d^{\Pn,k},\M^\Pn \right) \limit^{\WGP} (G^{\Theta,k},d^{\Theta,k},\p^{\Theta,k}).  \] 
\item[(d)] If $\Tn \ply \Theta$,  $\mu^{\Theta_n}=\infty$ for every $n\in \N$, and $\mu^\Theta=\infty$ then
\[(G^{\Tn,k},d^{\Tn,k},\p^{\Tn,k}) \limit^{\WGP} (G^{\Theta,k},d^{\Theta,k},\p^{\Theta,k}).  \]
\end{compactitem}
\end{theorem} 
%In passing, we prove by using Lemma \cite{EquivGP2}, that "cuts" behaves like uniform vertices:
%\begin{proposition} \label{rebranch} Let $\Theta\in \OmegaT$ such that $\mu^\Theta=\infty$. Let $(A^\Theta_i)$ be a family of i.i.d.random variables with law $p^\Theta$. Then we have the following equality in distribution:
%\[ (d^\Theta(Y_i^\Theta, Y_j^\Theta))_{i,j\in \N} =^{(d)} (d^\Theta(A_i^\Theta, A_j^\Theta))_{i,j\in \N} \]
%\end{proposition}
%\subsection{Notations}
%In this section we introduce some aditionnal notations that will be used throughout the paper. 
%For every distribution of degree $\D\in \Omega_{\D}$ we define the following variable. For every $1\leq i \leq n^\D$ let
%Similarly for every sequence $\P\in \Omega_{\P}$ let $\n^\P$ be the smallest integer in $\N\cup\{+\infty\}$ such that for every $n>\n^\P$, $p^{\P}_n=0$ and let $p_1^\P, p_2^{\P} \dots, p_{\infty}^\P$ be such that $\P=(p_1^\P, \dots, p_{\infty}^\P)$. Let $\{A^\P_i\}_{i\in \N}$ be a list of independent  identically distributed random variables such that for every $i\in \N\cup \{+\infty\}$ $\proba(A_1^\P=i)=p_i$. Then let $T^\P$ 
%Let $Y_1^\P<\dots$ be the time of repetitions that is index $i$ such that $A_i^\P\in \{A_1^\P,\dots,A_{i-1}^\P\}\backslash\{+\infty\}$ then for every $i\in \N$ let $Z_i^{\P}$ denotes the first index such that $A^\D_{Y_i^{\P}}=A^\D_{Z_i^{\P}}$.
\subsection{Gromov--Hausdorff convergence} \label{4.3}
GH convergence requires additional assumptions. In \cite{Uniform} we give quantitative assumptions. Here, we simply state rudimentary assumptions. We proved in Section 7.3 of \cite{Uniform} that the assumptions of \cite{Uniform} imply the followings. To simplify the notations, for every tree (and every $\R$-tree) $T$ and $v_1,\dots, v_a\in T$, we write $T(\{v_i\}_{1\leq i\leq a})$ for the subtree spanned by $v_1, \dots, v_a$.
\begin{Hypo} \label{Hypo3} For every $\e>0$,
\[ \lim_{a \to \infty}  \limsup_{n\to +\infty} \proba \left (\lambda^\Dn d_H \left (T^\Dn(\{\star_i\}_{0\leq i \leq a}),T^\Dn \right )>\e  \right )=0, \]
%where by convention, for every $x\in \R^+$, $T_x^\D$ denotes $T_n^\D$ where $n:=\inf\{i\in \N, i\leq \min(x,\n^\D)\}$.
\end{Hypo} 
\begin{Hypo} \label{Hypo3P} For every $\e>0$,
\[ \lim_{a \to \infty}  \limsup_{n\to +\infty} \proba \left (\s^\Pn d_H \left (T^\Pn(\{\star_i\}_{0\leq i \leq a}),T^\Pn \right )>\e  \right )=0, \]
%where by convention, for every $x\in \R^+$, $T_x^\P$ denotes the tree $T_n^\P$ where $n:=\inf\{i\in \N, i\leq x\}$.
\end{Hypo} 
\begin{Hypo} \label{Hypo3T}  For every $\e>0$,
\[ \lim_{a \to \infty}  \limsup_{n\to +\infty} \proba \left (d_H \left (T^\Tn(\{Y^{\Tn}_i\}_{0\leq i \leq a}),T^\Tn \right )>\e  \right )=0. \]
\end{Hypo} 
%
%We now state the main result of the section. 
\pagebreak[2]
\begin{theorem} \label{THM2} The following convergences hold weakly for the GH-topology. 
\begin{compactitem}
\item[(a)] If $\Dn\ply \Theta$, $\M^\Dn \to 0$, and Assumption \ref{Hypo3} is satisfied then
\[ \left ( G^{\Dn,k}, \lambda^\Dn d^{\Dn,k} \right) \limit^{\WGH} (G^{\Theta,k},d^{\Theta,k}).  \]
\item[(b)]  If $\Pn \ply \Theta$, $\M^\Pn \to 0$, and Assumption \ref{Hypo3P} is satisfied then
\[ \left ( G^{\Pn,k},\s^\Pn d^{\Pn,k} \right) \limit^{\WGH} (G^{\Theta,k},d^{\Theta,k}).  \] 
\item[(c)] If $\Tn \ply \Theta$, and Assumption \ref{Hypo3T} is satisfied then
\[(G^{\Tn,k},d^{\Tn,k}) \limit^{\WGH} (G^{\Theta,k},d^{\Theta,k}).  \]
\end{compactitem}
\end{theorem}
\begin{remark} $\bullet$ Unlike the assumptions of \cite{Uniform}, Assumption \ref{Hypo3}, \ref{Hypo3P}, \ref{Hypo3T} are sufficient and necessary. \\
$\bullet$ By \cite{Uniform} Lemma 4, one can deduce the GHP convergence from the GP and GH convergence and the fact that, since $p^\Theta$ have a.s. full support on $\T^\Theta$ (see \cite{ICRT1}), $p^{\Theta,k}$ have a.s. full support on $G^{\Theta,k}$.
\end{remark}
\section{Study of the bias} \label{Bias}
%In the whole section $\D$-trees have leaves $(\star_i)_{0\leq i \leq N^\D+1}$.
\subsection{Proof of Proposition \ref{MainMainMain} in the typical case} \label{Bias1}
Recall that for every $x,m\in \R^+$, $h_m=x\1_{x\geq m}$. Recall the definitions of $(\square_i)_{1\leq i \leq k}$ and $\square_{\square,k}$ from section \ref{DkDef}. For every $\D\in \OmegaD$ with $N^\D\geq 2k$ and $m\in \R^+$ let
\[f^\D(m):=\E\left [h_{m}\left (\frac{\square_{\square,k}(T^{\D})}{(\lambda^\D)^k}\right )\right ].\]
In this section we estimate $f^\D$ under the additional assumption $2N^\D\geq  \n^\D/\s^\D$, which is satisfied when there are not too many vertices with degree 2. 
\begin{proposition} \label{MainMain}  There exists $c,C>0$ such that for every $\D\in \OmegaD$ with $N^\D\geq \max(2k,\n^\D/(2\s^\D))$, and $m>0$, we have $f^\D(m)\leq Cm^{-c}$.
%\[ f^\D(\e)=\E\left [\frac{\1_{(\lambda^\D)^k \square_{\square,k}(T^{\D})\leq \e}}{(\lambda^\D)^k\square_{\square,k}(T^{\D})} \right ]\leq C \e^{c}. \]
\end{proposition} 
Our proof is organized as follow: We first upper bound $\square_{\square,k}$. Then we use H\"older's inequality to upper bound $f^\D(\e)$ with the numbers of leaves in some open balls around $\star_0$. Then we use Algorithm \ref{D-tree} to upper bound those numbers with $(Y_i)_{1\leq i \leq k}$. Finally we use the  continuum $\D$-tree construction of \cite{Uniform} to study $(Y_i)_{1\leq i \leq k}$ through random Poisson point process.

Let $d^\D$ be the graph distance in $T^\D$. Let $d'^\D(\cdot, \cdot):=\lambda^\D d^\D(\cdot,\cdot)$. We have:
\begin{lemma} \label{MIAOUH1} Let $C=2^{2k}(k+1)!$. For every $\e>0$, for every $\D\in \OmegaD$ with $N^\D\geq 2k$,
\[ f^\D(C\e^{-k})/(kC)\leq g^\D(\e):=\E\left [\frac{\1_{d'(\star_{1},\star_{2})\leq \e}}{\prod_{i=1}^k  d'(\star_{2i-1},\star_{2i})} \right ]. \]
%$1\leq i \leq k$ and tree $T$ which admit $\{\star_i\}_{1\leq i \leq 2k}$ as leaves,
\end{lemma} 
\begin{proof} First by definition of $\square_{\square,k}$, $\square_{\square,k}\leq (k+1)!2^k/ \prod_{i=1}^k \square_i$. Then note for every $1\leq i \leq k$ that $\square_i(T^\D)\geq d(\star_{2i-1},\star_{2i})-1\geq d(\star_{2i-1},\star_{2i})/2$. Indeed, the path between the father of $\star_{2i-1}$ and the father of $\star_{2i}$, together with the edge connecting those two fathers, forms a cycle. Thus,
\[ f^\D(C\e^{-k})/C\leq \E\left [\frac{\1_{\prod_{i=1}^k  d'(\star_{2i-1},\star_{2i})\leq \e^k}}{\prod_{i=1}^k  d'(\star_{2i-1},\star_{2i})} \right ] . \]
The desired result then follows from the symmetry of the leaves $(\star_i)_{1\leq i \leq 2k}$. (That is the fact that permuting the label of the leaves of $T^\D$ independently of $T^\D$ does not change the law of $T^\D$.) % $\alpha$  of $\{\star_i\}_{0\leq i \leq N^\D+1}$, $T^\D$ is also a $\D$-tree with leaves $(\star_{\alpha(i)})_{0\leq i \leq N^\D+1}$
\end{proof}
For the rest of the section $\e>0$, and $\D$ are fixed. We have to estimate $\prod_{i=1}^k d'(\star_{2i-1},\star_{2i})$. However, it is hard to estimate since it depends on $k$ separate parts of the tree. For this reason, we instead upper bound $g$ with the numbers of leaves in some open balls around $\star_0$. For every $n\geq 1$, let $M_n$ be the proportion of leaves $L\in T\backslash\{\star_0\}$ such that $2^{-n-1}< d'(\star_0,L)\leq 2^{-n}$ and let $M_0$ be the proportion such that $d'(\star_0,L)> 1/2$. Let $K_\e:=\inf\{n\in \N,2^{-n}\geq \e\}$. We have:
\begin{lemma} \label{bias1} There exists $C>0$ which depends only on $k$ such that,
\[ g(\e)\leq C \E\left [ \sum_{n=K_\e}^\infty 2^{nk} n^{3k}M^k_n \right ]^{1/k} \E\left [ \sum_{n=0}^\infty 2^{nk} n^{3k} M_n^k\right ]^{(k-1)/k}. \]
\end{lemma} 
\begin{proof}  In this proof $C$ denotes a real depending only on $k$ which may vary from line to line. First, let $(L_i)_{1\leq i\leq 2k}$ be uniform random variables  in $\{\star_i\}_{0\leq i \leq N+1}$. Note that by symmetry of the leaves, %that is the fact that for any permutation $\alpha$  of $\{\star_i\}_{0\leq i \leq N^\D+1}$, $T^\D$ is also a $\D$-tree with leaves $\{\star_{\alpha(i)}\}_{0\leq i \leq N^\D+1}$
\[ g(\e) = \E \left [ \left .   \frac{\1_{d'(L_{1},L_{2}) \leq \e}}{\prod_{i=1}^k d'(L_{2i-1},L_{2i})}  \right | \forall  i\neq j, L_i\neq L_j\right ]. \]
Then by roughly speaking slightly changing $(L_{i})_{1\leq i \leq 2k}$ such that some equalities may hold,
\begin{align*} g(\e) & \leq  C\E \left [\left . \frac{\1_{d'(L_{1},L_{2}) \leq \e}}{\prod_{i=1}^k d'(L_{2i-1},L_{2i})}\right | \forall 1\leq i \leq k, L_{2i-1}\neq L_{2i} \right ]. 
\\ & = C\E \left [ \E\left [ \left . \frac{\1_{d'(L_{1},L_{2}) \leq \e}}{d'(L_1,L_2)} \right | L_1,L_1\neq L_2, T \right ] \prod_{i=2}^{k} \E\left [  \left . \frac{1}{d'(L_{2i-1},L_{2i})} \right | L_{2i-1},L_{2i-1}\neq L_{2i},T \right ] \right ]. 
\end{align*}

Furthermore, by H\"older's inequality, and by symmetry of the leaves,
\begin{align} g(\e)^k&  \leq   C\E \left [ \E\left [ \left . \frac{\1_{d'(L_{1},L_{2}) \leq \e}}{d'(L_1,L_2)} \right | L_1,L_1\neq L_2,T \right ]^k \right ] \prod_{i=2}^{k} \E\left [\E\left [  \left . \frac{1}{d'(L_{2i-1},L_{2i})} \right | L_{2i-1},L_{2i-1}\neq L_{2i},T \right ]^k \right ] \notag
\\ & =  C\E \left [ \E\left [ \left . \frac{\1_{d'(\star_0,L_{2}) \leq \e}}{d'(\star_0,L_2)} \right |\star_0\neq L_2, T \right ]^k \right ]  \E\left [\E\left [  \left . \frac{1}{d'(\star_0,L_{2})} \right |\star_0\neq L_{2}, T \right ]^k \right ]^{k-1}. \notag \end{align}
Therefore, we have by definition of $(M_n)_{n\in \{0\}\cup \N}$,
\begin{equation} g^\D(\e)^k\leq C \E\left [ \left ( \sum_{i=K_\e}^\infty 2^n M_n \right )^k \right ] \E\left [ \left (\sum_{i=0}^\infty 2^n M_n\right )^k \right ]^{k-1}. \label{OUAFOUAF} \end{equation} 
If $k=1$ the desired results follow from \eqref{OUAFOUAF}. If $k\geq 2$ then we have a.s., by H\"older's inequality,
\[ \sum_{i=K_\e}^\infty 2^n M_n  \leq \left(\sum_{i=K_\e}^\infty \left ( 2^n n^3 M_n \right )^k \right )^{1/k} \left(\sum_{i=N}^\infty \left ( \frac{1}{n^3} \right )^{k/(k-1)} \right )^{(k-1)/k} , \]
and similarly for $\sum_{i=0}^\infty  2^n M_n$. And the desired result follows from \eqref{OUAFOUAF}. %lemma directly follows (the bound is trivial for $k=1$)
%\begin{equation} f_\e(D)^k\leq C \E\left [ \sum_{i=N}^\infty 2^{nk} n^{3k}M^k_n \right ] \E\left [ \sum_{i=0}^\infty 2^{nk} n^{3k} M_n^k\right ]^{k-1}. \label{2110b} \end{equation}
\end{proof}
Recall Section \ref{2.2}. We now upper bound for $n\in \N$, $\E[M_n^k]$ using Algorithm \ref{D-tree}. Recall the definition of $A^\D$. Let $Y_1, Y_2,\dots$ be the indexes such that $A^\D_i\in \{A^\D_1,\dots, A^\D_{i-1}\}$.
\begin{lemma} \label{masse d'une truite} For every $n\in \N$, 
\[ \E[M_n^k]\leq k^k \sum_{a=1}^k \frac{1}{N^{k-a}} \proba \left (Y_a\leq \frac{a}{2^n} \frac{s}{ \sigma} \right ).\]
\end{lemma} 
\begin{proof} First, let $(L_i)_{1\leq i\leq 2k}$ be uniform random variables  in $\{\star_i\}_{1\leq i \leq N+1}$. By definition of $M_n^k$,
\begin{equation*} \E[M_n^k] =\proba\left [\frac{1}{2^{n+1}} < d'(\star_0,L_1),\dots, d'(\star_0,L_k)\leq \frac{1}{2^n}\right ]. \end{equation*}
%\notag
%\\ & =  C\E \left [ \E\left [ \left . \prod_{i=1}^k \frac{\1_{d'(\star_0,L_{i}) \leq \e}}{d'(\star_0,L_i)} \right | T \right ] \right ] \E\left [\E\left [  \left . \prod_{i=1}^k \frac{1}{d'(\star_0,L_i)} \right | T \right ] \right ]^{k-1} \notag
%\\ & =  C\E\left [ \prod_{i=1}^k \frac{\1_{d'(\star_0,L_{i}) \leq \e}}{d'(\star_0,L_i)} \right ] \E\left [\prod_{i=1}^k \frac{1}{d'(\star_0,L_i)} \right ]^{k-1} . \label{2110a} 
Then we want distinct leaves to use Algorithm \ref{D-tree}. To this end, we develop the right hand side above by distinguishing the cases of equality. Let $\mathfrak{P}(k)$ be the set of partition of $\{1,\dots, k\}$. %that is the set of sets $I=\{I_1,\dots, I_a\}$ such that $I_1,\dots, I_a$ are disjoint sets with union $\{1,\dots, k\}$.
For every $I=\{I_1,\dots, I_a\}\in \mathfrak{P}(k)$, let $\AAA_I$ be the event that for every $x,y \in \{1,\dots, k\}$, $L_x=L_y$ iff they are in the same $I_i$.  For every $I\subset \{1,\dots, k\}$ let $m_{I}:=\min(I)$.  We have,
\begin{align*}& \E\left [M_n^k \right ]  = \sum_{I=\{I_1,\dots, I_a\}\in \mathfrak{P}(k)} \proba \left [ \AAA_I \,, \,  \frac{1}{2^{n+1}} < d'(\star_0,L_1),\dots, d'(\star_0,L_k)\leq \frac{1}{2^n}\right ]
\\ & =\sum_{\{I_1,\dots, I_a\}\in \mathfrak{P}(k)}  \frac{1}{(N+1)^{k-a} } \proba \left [L_{m_{I_1}}\neq \dots \neq L_{m_{I_a}} \, , \, \frac{1}{2^{n+1}} < d'(\star_0,L_{m_{I_1}}),\dots, d'(\star_0,L_{m_{I_a}})\leq \frac{1}{2^n} \right ].  \end{align*}
Then by symmetry of the leaves,
\begin{equation*}  \E\left [M_n^k \right ]  =  \sum_{\{I_1,\dots, I_a\}\in \mathfrak{P}(k)}  \frac{1}{(N+1)^{k-a}} \proba \left [\frac{1}{2^{n+1}} <d'(\star_0,\star_1),\dots, d'(\star_0,\star_a)\leq \frac{1}{2^n}\right ]. \end{equation*}
So since there is at most $k^k$ partitions of $\{1,\dots, k\}$, 
\begin{equation}  \E\left [M_n^k \right ]  \leq k^k\sum_{a=1}^k  \frac{1}{N^{k-a}} \proba \left [\frac{1}{2^{n+1}} < d'(\star_0,\star_1),\dots, d'(\star_0,\star_a)\leq \frac{1}{2^n}\right ].\label{KProject} \end{equation}

Finally we use Algorithm \ref{D-tree}. It is direct from the construction that, writing $Y_0=0$,
\[ Y_a= \sum_{i=1}^a (Y_i-Y_{i-1})\leq  \sum_{i=1}^a (d(\star_0,\star_i)-1)\leq (\n/\s)\sum_{i=1}^a d'(\star_0,\star_i). \]
So the desired results follows from \eqref{KProject}.
\end{proof}
We now upper bound $Y_a$ using a part of the continuum $\D$-tree construction of \cite{Uniform}:
\begin{compactitem}
\item[-] Let $(X_i)_{1\leq i \leq \n}$ be a family of independent exponential random variables of parameter $(d_i/\sigma)_{1\leq i \leq \n}$.
\item[-] Let $\mu$ be the measure on $\R^+$ defined by $\mu=\sum_{i=1}^{\n} \delta_{X_i} \left (d_i-1 \right )/\sigma$. 
\item[-] Let $(\hat Y_i)_{i\in \N}$ be a Poisson point process  on $\R^+$ of rate $\mu[0,y]dy$.
\item[-] Let $(E_i)_{1\leq i \leq \n-1}$ be a family of exponential random variables of mean $(\sigma/(\n-i))_{1\leq i \leq \n-1}$.
\end{compactitem}
By \cite{Uniform} Lemma 10 there exists a coupling such that $Y_a$ is independent of $(E_i)_{1\leq i \leq \n-1}$ and such that a.s. $\sum_{i=1}^{Y_a} E_i\leq \hat Y_a$. Moreover, we have:
\begin{lemma}  \label{truecutbound} For every $a,n\in \N$ with $n\leq \n/2$,
\[ \proba \left (Y_a\leq n\right ) \leq  \proba (\hat Y_a\leq 4n\sigma/\n  )/2 .\]
\end{lemma}
\begin{proof} Fix $n\leq \n/2$. It is easy to check from basic estimates on the Gamma distribution that,
\begin{align*} \proba\left ( \sum_{i=1}^{n} E_i \leq  4 n(\sigma/\n) \right  )  \geq 1/2.  \end{align*}
So since $Y_a$ and $(E_i)_{1\leq i \leq \n-1}$ are independent,
\begin{equation*} \proba \left (\hat Y_a\leq 4n\frac{\sigma}{\n} \right )  \geq \proba \left (\sum_{i=1}^{Y_a} E_i \leq 4n\frac{\sigma}{\n} \right ) \geq \proba \left (Y_a \leq n, \sum_{i=1}^{n} E_i \leq 4n\frac{\sigma}{\n} \right  )\geq  \frac{1}{2} \proba(Y_a \geq n). \qedhere \end{equation*}
\end{proof}
Hence, to upper bound $Y_a$ it is enough to upper bound $\hat Y_a$. To this end, we first upper bound $\mu$.
\begin{lemma} \label{dumb} For every $a\in \N$,  \begin{compactitem}
\item[ (a)] For every $x,t>0$, $\proba( \mu[0,x]> t) \leq x/t$.
\item[(b)] For every $0\leq x\leq 1\leq t$, $\proba( \mu[0,x]> t) \leq e^{-t/4}$.
\end{compactitem}
\end{lemma} 
\begin{proof} Note that by definition of $\mu$, $(X_i)_{1\leq i \leq \n}$ and $\sigma$,
\begin{equation*} \E[\mu[0,x]]=\sum_{i=1}^\s \frac{d_i-1}{\sigma} \proba(X_i>x)\leq \sum_{i=1}^\s \frac{d_i-1}{\sigma} \frac{xd_i}{\sigma} \leq x.  \end{equation*}%\label{2410}
So (a) follows from Markov's inequality. Also $\mu[0,x]$ is a sum of independent random variables bounded by 1 so (b) follows from Bernstein's inequality (see \cite{Massart} Section 2.8).
%
 %Fix $\delta>0$. Let $S_{\leq}:=\{i\in \N, (d_i-1)/\sigma \leq \delta \}$ and let $S_{>}:=\{i\in \N, (d_i-1)/\sigma >\delta\}$. We split $\mu[0,x]$ into  
%\[ \mu_{\leq}:=\sum_{i\in S_{\leq} } \frac{d_i-1}{\sigma}\1_{X_i\leq x} \quad  \text{and} \quad \mu_{>}:=\sum_{i\in S_{>}} (d_i-1 )/\sigma\1_{X_i\leq x},\]
% and bound each sum. First note that by definition of $\{X_i\}$,
%\[\proba(\mu_{>x}= 0) = \prod_{i\in S_>}\proba(X_i>x)  =\prod_{i\in S_>}e^{-d_i/\sigma x}=e^{-\sum_{i\in S_>} d_i/\sigma x}.\]
%Then by definition of $S_>$ and $\sigma^2=\sum_{i=1}^\n d_i(d_i-1)$, we have,
%\[x \sum_{i\in S_>} d_i/\sigma \leq x  \sum_{i\in S_>}  \frac{d_i(d_i-1)}{-\sigma^2 \delta } \leq x/\delta. \]
%Therefore,
%\begin{equation} \proba(\mu_{>}\neq  0)\leq 1-e^{-x/t}\leq x/\delta \label{WTF} \end{equation}
%
%On the other hand, $\mu_\leq$ is a sum of independent positive random variable bounded by $\delta$. And,
%\[ \E[ \mu_\leq ] \leq \sum_{i\in \N} \frac{d_i-1}{\sigma}\proba(X_i\leq x)\leq \sum_{i\in \N} \frac{d_i-1}{\sigma} \frac{d_ix}{\sigma}=x.\]
%So by Bernstein's inequality (see e.g. \cite{Massart} Section 2.8) for every $\lambda \geq 0$,
%\begin{equation} \proba(\mu_{\leq} \geq \sqrt{2\delta x \lambda }+\delta \lambda ) \leq e^{-\lambda }, \label{WTF2} \end{equation}
%
%Finally by \eqref{WTF} and \eqref{WTF2}, we have, for every $x,\delta, \lambda >0$,
%\[ \proba(\mu[0,x]\geq \sqrt{2\delta x \lambda }+\delta \lambda)\leq x/\delta+ e^{-\lambda }, \]
%And the desired results then follows by letting $\lambda:=-\log(x/t)$ and $\delta:=\frac{t}{-2\log(x/t)}$.
\end{proof} 
\begin{lemma} \label{falsecutbound} For every $a\in \N$ and $0\leq x\leq e^{-9}$, $\proba (\hat Y_a\leq x ) \leq 3x^{a+1}(-4a \log x)^a$.
\end{lemma} 
\begin{proof} %If $a=1$ then the result follows from \eqref{2410}, so we may assume $a\geq 2$. 

By definition of $( \hat Y_i)_{i\in \N}$, conditionally on $\mu$, $\max \{i\in \N, Y_i\leq x\}$ is a Poisson random variable of mean $\int_0^x \mu[0,t]dt \leq x\mu[0,x]$. So, by basic inequalities on the Poisson distributions,
\begin{equation} \proba(\hat Y_a\leq x)= \E[ \proba(\hat Y_a\leq x|\mu) ] \leq \E[(x\mu[0,x])^a ]. \label{2310} \end{equation}

Then we have by integration by part and Lemma \ref{dumb}, 
\begin{align*} \E[\mu[0,x]^a] & = \int_0^\infty \proba(\mu[0,x] \geq t) (a t^{a-1} dt)
\\ &  \leq \int_0^x a t^{a-1} dt+ \int_x^{-4\log x} (x/t) (a t^{a-1} dt) + \int_{-4 \log x}^\infty e^{-t/4} (a t^{a-1} dt)
\\ & \leq 3x(-4a \log x)^a,
 \end{align*}
using basic calculus for the last inequality. This concludes the proof.
 %Then by elementary computation, 
 %\[ \int_{3x}^\infty e^{-(t-x)^2/(2t)} (a t^{a-1} dt)\leq \int_{3x}^\infty e^{-t/3} (a t^{a-1} dt)\leq a!3^a \int_{3x}^\infty e^{-t/3}, \]
 \end{proof}
 \begin{proof}[Proof of Proposition \ref{MainMain}] We now complete our upper bound for $f(\D)$. In this proof, $c, C$ denote reals which depend only on $k$ and which may vary from line to line. First by Lemmas \ref{falsecutbound} and \ref{truecutbound} we have for every $1\leq a \leq k$ and $0\leq x <1/16$,
 \[  \proba \left (Y_a\leq x\n/\s \right )\leq Cx^{a+1}(-\log(x))^c. \]
 
 Then by Lemma \ref{masse d'une truite}, and $2N\geq \n/\s$, for every $n\in \N$ with $(\n/\s)/2^n \geq 1$, 
 \begin{align} \E[M_n^k] & \leq  k^k \sum_{a=1}^k \frac{1}{N^{k-a}} \proba \left (Y_a\leq \frac{a}{2^n} \frac{s}{ \sigma} \right )\notag
 \\ & \leq k^k \sum_{a=1}^k \left (\frac{2\s}{\n } \right )^{k-a} C \left ( \frac{a}{2^n} \right)^{a+1} n^c  \notag
 \\ & \leq \frac{Cn^c}{2^{(k+1)n}}. \label{rhume d'enfant}  \end{align}
 Note that \eqref{rhume d'enfant} naturally extends to the $n\in \N$ with $(\n/\s)/2^n < 1$ since for those $n$ almost surely for $1\leq a \leq k$, we have $Y_a\geq a >\frac{a}{2^n} \frac{s}{ \sigma}$.
 
 Next, since $K_\e=\inf\{n\in \N,2^{-n}\geq \e\}$,
 \[  \E\left [ \sum_{i=K_\e}^\infty 2^{nk} n^{3k}M^k_n \right ] \leq C2^{-K_\e c} \leq C(2\e)^c, \]
 and 
  \[  \E\left [ \sum_{i=0}^\infty 2^{nk} n^{3k}M^k_n \right ] \leq C.\]
Thus by Lemma \ref{bias1}, 
\begin{equation} g^\D(\e)\leq C\E\left [ \sum_{i=0}^\infty 2^{nk} n^{3k}M^k_n \right ]^{1/k} \E\left [ \sum_{i=K_\e}^\infty 2^{nk} n^{3k}M^k_n \right ]^{(k-1)/k}  \leq C (2\e)^c. \label{BlueArchive} \end{equation}
Finally, Proposition \ref{MainMain} follows from Lemma \ref{MIAOUH1}.
 \end{proof}
 Along the way by \eqref{BlueArchive} we have the following result, which we extend in the next section.
 \begin{lemma} \label{BlueArchive2} There exists $c,C>0$ which depends only on $k$ such that for every $\e>0$, $\D\in \OmegaD$ with $N^\D\geq \max(2k,\n^\D/(2\s^\D))$, $g^\D(\e)\leq C \e^c$.
 \end{lemma}
 \subsection{Proof of Proposition \ref{MainMainMain} when there are many vertices of degree 2}  \label{Bias2}
 This section is organized as follow. We first detail how to remove or add vertices of degree 2. We then prove from those constructions a connection between the $\D$-trees that do not have any vertice of degree 2 and the others. Finally we use this connection to prove Proposition \ref{MainMainMain}.

%\begin{figure}[!h] \label{Maison}
%\centering
%\includegraphics[scale=0.5]{Maison.eps}
%\caption{A graph $G$ with $\nabla(G)$. Edgepoints are red.} \label{Maison}
%\end{figure}%
 First for every graph $G$ and $x\in G$, we call $x$ an edgepoint if $x$ have degree 2. 
 %a leaf if $x$ have degree 1, an edgepoint if $x$ have degree 2, and a branchpoint if $x$ have degree at least 3. 
 A simple way to remove the edgepoints is to shortcut them: Formally if $T=(V,E)$ is a tree, then $\nabla T$ be the tree $(V',E')$ such that $V'=\{v\in V,\deg_T(v)\neq 2\}$ and for every $v,w\in V'$, $\{v,w\}\in E'$ iff there exists a path between $v$ and $w$ that only pass by $v$, $w$ and vertices of degree $2$. Note that $\nabla$ keep the degrees: for every $v\in T$ with $\deg_v(T)\neq 2$, we have $v\in \nabla T$ and $\deg_T(v)=\deg_{\nabla T}(v)$. %So $\nabla T$ is a tree.
 \begin{remark} One may extends $\nabla$ to general graph. However, the natural way to preserves the degrees is to work with multigraph. We avoid this issue by working with trees.
 \end{remark}
 
 % Formally, for every graph $G=(V,E)$ and edgepoint $x\in V$, let $v_1,v_2\in V$ be the only vertices such that $\{v_1,x\},\{v_2,x\}\in G$ and let 
 %\[ \nabla_x(G):=(V\backslash x, E\cup \{ \{v_1,v_2\} \}\backslash \{\{x,v_1\},\{x,v_2\}\}). \] 
 %Also, if $x_1,\dots x_a$ are the edgepoints of $G$ let (see Figure \ref{Maison})
 %\[ \nabla G:=\nabla_{x_1}\circ \nabla_{x_2}\dots  \circ \nabla_{x_a}(G). \]
% Note that $\nabla G$ does not depends on the order in which $x_1,\dots, x_a$ are taken. Moreover, note that $\nabla G$ does not have any edgepoint, and that the degrees of the other vertices are preserved.
  
 Reciprocally one may construct any tree by adding some edgepoints along the oriented edges of a tree without edgepoint: For every $T=(V,E)$ let $( \vec e_i(T))_{1\leq i \leq \# E}$ be some fixed oriented edges of $T$ such that each edge of $E$ appears in one and only one direction. Let $((W_{i,j})_{1\leq i \leq r_i})_{1\leq i \leq \# E}$ be some vertices that are not in $V$. For every $1\leq i \leq \#E$ let  $(W_{i,0},W_{i,r_i+1}):= \vec e_i(G)$. Let
 \[ \Delta(T,((W_{i,j})_{1\leq i \leq r_i})_{1\leq i \leq \# E}:=\left (V\cup \{W_{i,j}\}_{1\leq i \leq \#E, 1\leq j \leq r_i}, \{ \{W_{i,j},W_{i,j+1}\} \}_{1\leq i \leq \#E, 0\leq j \leq r_i} \right ).  \]
 %Intuitively, we add some vertices along $\{ \vec e_i(G)\}_{1\leq i \leq \# E}$ thus splitting each edges into many others.

 We now use $\Delta, \nabla$ to study $\D$-trees. Beforehand, let us introduce some notations. For every $\D=(d_1,\dots, d_\n)\in \Omega_\D$, let $\n_{\geq 2}^\D:=\#\{a\in \N, d_a\geq 2\}$, let $\n_{\geq 1}^\D=\#\{a\in \N, d_a\geq 1\}$, and let $\n^\D_1:=\#\{a\in \N, d_a=1\}$. Also let $\nabla \D$ be the sequence $(d_1,d_2,\dots, d_{\n_{\geq 2}}, d_{\n_{\geq 1}+1},\dots, d_\n)$. 
 
  Also we say that $((W_{i,j})_{1\leq j \leq r_i})_{1\leq i \leq n}$ is an ordered partition of size $n\in \N$ of a finite set $E$ iff for $1\leq i \leq n$, $r_i\in \{0\} \cup \N$, and $(i,j)\mapsto W_{i,j}$ is a bijection from $\{1\leq i \leq n, 1\leq j \leq r_i\}$ to $E$.
We have the following connections between $\D$-trees and $\nabla \D$-trees:
 
  \begin{lemma} \label{Nabla1} Let $\D\in \OmegaD$. Let $W$ be a  uniform ordered partition of size $\n^{\nabla \D}-1$ of $\{V_i\}_{i:d^\D_i=1}$. Then,
a) $\nabla (T^\D)$ is a $\nabla \D$-tree, and b) $\Delta(T^{\nabla \D},W)$ is a $\D$-tree 
% \begin{compactitem}
%\item[a)] 
%\item[b)] .
%\end{compactitem}
\end{lemma}
\begin{proof} First note that $\nabla ( \Delta(T^{\nabla \D},W))=T^{\nabla \D}$, since this tree is obtained by adding some edgepoint on $T^{\nabla \D}$, which do not have edgepoint, then by removing all edgepoint. So b) imply a). 

Toward b), simply note that $\Delta$ may be seen as a bijection from trees with degree sequence $\nabla \D$ and ordered partition of size $s^{\nabla \D}-1$ of  $\{V_i^\D\}_{d_i=1}$ toward trees with degree sequence $\D$. (Indeed, one may recover the initial tree by applying $\Delta$ and then read the ordered partition by, roughly speaking, following each oriented edges of the initial tree on the image tree.)
\end{proof} 
We now prove Proposition \ref{MainMainMain}. To this end, it is enough to remove the assumption $2N^\D\geq  \n^\D/\s^\D$ of Proposition \ref{MainMain}. Note that it is satisfied when $s_{1}^\D=0$ since in this case, $\sigma^\D\geq 1$ and $\n^\D=N^\D+\n_{\geq 2}^\D\leq 2N^\D$. For this reason, our goal for the rest of the section will be to prove the following result, which together with Lemmas \ref{BlueArchive2} and \ref{MIAOUH1} yields Proposition \ref{MainMainMain}.

\begin{proposition} \label{MainMain'} Recall the definition of $g$ from Lemma \ref{MIAOUH1}. There exists $C>0$, which depends only on $k$, such that for every $\D\in \OmegaD$ with $N^\D\geq 2k$ and $\e>0$, 
\[ g^\D(\e)\leq C\e \left  (\int_{\e}^1 g^{\nabla \D}(\delta)/\delta^2 d\delta+kg^{\nabla \D}(1)+1 \right ). \]
\end{proposition} 
To this end, it is enough to lower bound $(d^\D(\star_{2i-1},\star_{2i}))_{1\leq i \leq k}$ using $(d^{\nabla\D}(\star_{2i-1},\star_{2i}))_{1\leq i \leq k}$. To do so, by Lemma \ref{Nabla1} (b), it suffices to study uniform ordered partitions. More precisely, we have to lower bound the cardinal of the sets of those partitions, which corresponds to the numbers of edgepoint added on each edge. This is done in the following lemma.

\begin{lemma}  \label{Nabla2}Let $((W_{i,j})_{1\leq j \leq R_i})_{1\leq i \leq n}$ be a uniform ordered partition of size $n$ of a finite set $E$.
\begin{compactitem}
\item[(a)] $(R_i)_{1\leq i\leq n}$ is uniform among all set of integers such that $\sum_{i=1}^n R_i=\# E$.
\item[(b)] Let $(S_i)_{1\leq i\leq n}$ be independent geometric random variables of mean $\#E/n$ conditioned on $\sum_{i=1}^n S_i\leq \# E$. Then there exists a coupling between $(R_i)_{1\leq i \leq n}$ and $(S_i)_{1\leq i\leq n}$ such that almost surely for every $1\leq i \leq n$, $R_i\geq S_i$ .
\end{compactitem} 
\end{lemma} 
\begin{proof} Toward (a), simply note that given $(R_i)_{1\leq i\leq n}$, there are exactly $\# E!$ possible ways to label $((W_{i,j})_{1\leq j \leq R_i})_{1\leq i \leq n}$ to form an ordered partition of size $n$ of $E$. Then (b) is an easy exercise.
\end{proof}
Next, in order to use the independency of Lemma \ref{Nabla2} (b), we will use the following lemma:
\begin{lemma} \label{Nabla3}
Let $T$ be a tree. Assume that $(\star_i)_{1\leq i \leq 2k}$ are leaves of $T$. For every $1\leq i \leq k$ let $\mathcal{E}_i$ be the set of edges that are on the minimal path between $\star_{2i-1}$ and $\star_{2i}$. Then there exists $(\mathcal E'_i)_{1\leq i \leq k}$ disjoint subsets of $(\mathcal{E}_i)_{1\leq i \leq k}$ such that for every $1\leq i \leq k$, $\# \mathcal E'_i\geq \max(\#\mathcal E_i/k,2)$.
\end{lemma}
\begin{proof} Consider the following informal construction of $(\mathcal E'_i)_{1\leq i \leq k}$:
\begin{compactitem}
\item[-] First let for $1\leq i \leq k$, $\mathcal E'_i:= \{ \{\star_{2i-1},F_{2i-1}\}, \{\star_{2i},F_{2i}\}\}$, where for $1\leq i \leq 2k$, $F_{i}$ is the father of $\star_i$ in $T$.
\item[-] Then while $\bigcup_{i=1}^k \mathcal E'_i \neq \bigcup_{i=1}^k \mathcal E_i$: 
\begin{compactitem}
\item[-] For $1\leq i \leq k$: If possible add to $\mathcal E'_i$ an arbitrary edge in $\Ee_i$ that is not yet in $\bigcup_{j=1}^k \mathcal E'_i$.
\end{compactitem}
\end{compactitem} 
It is easy to check that $(\mathcal E'_i)_{1\leq i \leq k}$ are disjoint subsets of $(\mathcal{E}_i)_{1\leq i \leq k}$. Also for $1\leq i \leq k$, $\# \mathcal E'_i \geq 2$. Finally a quick enumeration gives that at the end of the algorithm $\#\mathcal E'_i \geq \#\mathcal{E}_i/k$. 
\end{proof} 
\begin{proof}[Proof of Proposition \ref{MainMain'}.] Let $\e>0$. Let $\D\in \OmegaD$. Let $W$ be a uniform ordered partition of size $\n_{\geq 2}^\D$ of $\{V_i^\D\}_{i:d_i=1}$ and independent of $T^{\nabla \D}$. Let $d^{\nabla D, W}$ be the graph distance on $\Delta(T^{\nabla \D},W)$. Then by Lemma \ref{Nabla1} (b), $\Delta(T^{\nabla \D},W)$ is a $\D$-tree. So, by definition of $g$, it is enough upper bound
%\[g^\D(\e)=\E\left [\frac{\1_{d'^{\nabla D, W}(\star_{1},\star_{2})\leq \e}}{\prod_{i=1}^k  d'^{\nabla D, W}(\star_{2i-1},\star_{2i})} \right ]. \]
%Thus it is enough to prove that almost surely, 
\begin{equation} G^{\D}(\e,T^{\nabla \D}):=\E\left [\left . \frac{\1_{\lambda^\D d^{\nabla \D, W}(\star_{1},\star_{2})\leq \e}}{\prod_{i=1}^k\left (\lambda^\D d^{\nabla \D, W}(\star_{2i-1},\star_{2i})\right )} \right | T^{\nabla \D}\right ] . \label{3011.17h} \end{equation}
 
To this end, let us use Lemmas \ref{Nabla2} and \ref{Nabla3}. Let $\Ee$ be the set of edges of $T^{\nabla \D}$. Let $(S_e)_{e\in \Ee}$ be independent geometric random variables of mean $\n_{1}^\D/\# \Ee$ conditioned on $\sum_{e\in \Ee} S_i\leq s^\D_1$.
For $1\leq i \leq k$ let $\Ee_i$ be the set of edges that are on the minimal path between $\star_{2i-1}$ and $\star_{2i}$ in $T^{\nabla \D}$. By definition of $\Delta$, and by Lemma \ref{Nabla2}, note that, there exists a coupling between $W$ and $(S_e)_{e\in \Ee}$ such that a.s. for $1\leq i \leq 2k$,
\begin{equation} d^{\nabla D, W}(\star_{2i-1},\star_{2i}) \geq \sum_{e\in \Ee_i}(1+S_e). \label{3011.18h}\end{equation}

Then,  by Lemma \ref{Nabla3}, let  $(\mathcal E'_i)_{1\leq i \leq k}$ be disjoint subsets of $(\mathcal{E}_i)_{1\leq i \leq k}$ such that for every $1\leq i \leq k$, $\# \mathcal E'_i\geq \max(\#\mathcal E_i/k,2)$. It directly follows from \eqref{3011.18h} that a.s. for $1\leq i \leq 2k$, 
\begin{equation*} d^{\nabla D, W}(\star_{2i-1},\star_{2i}) \geq \sum_{e\in \Ee'_i}(1+S_e) . \label{3011.18hb}\end{equation*}
Therefore, 
\[G^{\D}(\e,T^{\nabla \D})\leq \E\left [\left . \frac{\1_{\lambda^\D\sum_{e\in \Ee'_1}(1+S_e)\leq \e}}{\prod_{i=1}^k \left (\lambda^\D\sum_{e\in \Ee'_i}(1+S_e)\right )} \right | T^{\nabla \D}\right ] . \]

Hence, if $(S'_e)_{e\in \Ee}$ are independent geometric random variables of mean $\n_{1}^\D/\# \Ee$,
\[G^{\D}(\e,T^{\nabla \D})\leq \frac{1}{\proba\left (\left . \sum_{e\in \Ee} S'_e\leq \n_1^\D\right | T^{\Delta \D} \right )}\E\left [\left . \frac{\1_{\lambda^\D\sum_{e\in \Ee'_1}(1+S'_e)\leq \e}}{\prod_{i=1}^k \left (\lambda^\D\sum_{e\in \Ee'_i}(1+S'_e)\right )} \right | T^{\nabla \D}\right ]. \]
Then note that there exists a constant $C<\infty$ that does not depends on $k,\D$ such that a.s. 
$\proba\left (\left . \sum_{e\in \Ee} S'_e\leq \n_1^\D\right | T^{\Delta \D} \right ) \leq 1/C$. 
So, since $(\mathcal E'_i)_{1\leq i \leq k}$ are disjoint and $(S'_e)_{e\in \Ee}$ are independent,
\[G^{\D}(\e,T^{\nabla \D})\leq C\left (\lambda^\D\right )^{-k}\E\left [\left . \frac{\1_{\sum_{e\in \Ee'_1}(1+S'_e)\leq \e/\lambda^\D}}{\sum_{e\in \Ee'_1}(1+S'_e)} \right | T^{\nabla \D}\right ] \prod_{i=2}^k \E\left [\left . \frac{1}{\sum_{e\in \Ee'_i}(1+S'_e)} \right | T^{\nabla \D}\right ] . \]
Therefore we have using Lemma \ref{Concentretoi} below, and the fact that for every $1\leq i \leq k$, $\# \mathcal E'_i\geq 2$,
\begin{equation}G^{\D}(\e,T^{\nabla \D})\leq C(2e)^{-k}\left (\lambda^\D\right )^{-k} \min\left (1,\frac{e\e/\lambda^\D}{\#\Ee'_1(1+\n_{1}^\D/\# \Ee)}\right ) \prod_{i=1}^k \frac{1}{\#\Ee'_i (1+\n_{1}^\D/\# \Ee)}. \label{0112.11h} \end{equation}

Next, let us rewrite \eqref{0112.11h}. First, note that for every $1\leq i \leq k$,
\[ \#\Ee'_i \geq \#\Ee_i/k=d^{\nabla \D}(\star_{2i-1},\star_{2i})/k. \]
 Also, 
\[ 1+\frac{\n_{1}^\D}{\# \Ee}= 1+\frac{\n_{1}^\D}{\n^{\nabla \D}-1}=\frac{\n^{\nabla \D}+\n_{1}^\D-1}{\n^{\nabla \D}-1}=\frac{\n^{\D}-1}{\n^{\nabla \D}-1}\geq  \frac{\n^{\D}}{\n^{\nabla \D}}=\frac{\lambda^\D}{\lambda^{\nabla \D}},\]
noting for the last equality that $\sigma^\D=\sigma^{\nabla\D}$. Then by elementary calculus it is easy to prove that,
\begin{align*}  \min\left (1,\frac{e\e/\lambda^\D}{\#\Ee'_1(1+\n_{1}^\D/\# \Ee)}\right )& \leq ke \min\left (1,\frac{\e}{\lambda^{\nabla \D} d^{\nabla \D}(\star_1,\star_2)}\right ) 
\\ &=  ke \e \int_{\e}^\infty \1_{\lambda^{\nabla \D}d^{\nabla \D}(\star_1,\star_2)) \leq \delta} \frac{d\delta}{\delta^2}.
\end{align*}
Therefore by  \eqref{0112.11h},
\begin{equation*}G^{\D}(\e,T^{\nabla \D})\leq C(2ek)^{k+1}\left (\lambda^{\nabla \D}\right )^{-k}  \e \int_{\e}^\infty \1_{\lambda^{\nabla \D}d^{\nabla \D}(\star_1,\star_2) \leq \delta} \frac{d\delta}{\delta^2} \prod_{i=1}^k \frac{1}{d^{\nabla \D}(\star_{2i-1},\star_{2i})}. \end{equation*}
 
 Finally by taking the expectation and by Fubini's theorem, we have,
\[ \E[G^{\D}(\e,T^{\nabla \D})] \leq C(2ek)^{k+1}\left (\lambda^{\nabla\D}\right )^{-k}  \e \int_{\e}^\infty \E\left [ \1_{\lambda^{\nabla \D}d^{\nabla \D}(\star_1,\star_2)\leq \delta} \prod_{i=1}^k \frac{1}{d^{\nabla \D}(\star_{2i-1},\star_{2i})} \right ]  \frac{d\delta}{\delta^2}, \] 
which yields by definition of $G$ and $g$,
\begin{equation} g^\D(\e) \leq C(2ek)^{k+1}  \e \int_{\e}^\infty g^{\nabla \D}(\delta)  \frac{d\delta}{\delta^2}. \label{WOUAFWOUAF} \end{equation}

To conclude the proof, note that for $\delta\geq 1$,
\begin{align*} g^{\nabla \D}(\delta) & \leq \E\left [ \prod_{i=1}^k \frac{1}{\lambda^{\nabla \D}d^{\nabla \D}(\star_{2i-1},\star_{2i})} \right ] 
\\ & \leq \E\left [1+ \sum_{j=1}^k  \1_{\lambda^{\nabla\D}d^{\nabla \D}(\star_{2j-1},\star_{2j}) \leq 1} \prod_{i=1}^k \frac{1}{\lambda^\D d^{\nabla \D}(\star_{2i-1},\star_{2i})} \right ]=1+kg^{\nabla \D}(1).  \end{align*}
So the desired result follows from \eqref{WOUAFWOUAF}.
%
%And finally by taking the expectation and by Fubini's Theorem,
%\[ \E[G^{\D}(\e,T^{\nabla \D})] \leq C(2ek)^{k+1}\left (\frac{\n^{\nabla \D}}{\sigma^{\nabla \D}}\right )^k  \e \int_{\e}^\infty \E\left [ \1_{d^{\nabla \D}(\star_1,\star_2)(\sigma^{\nabla \D}/\n^{\nabla \D}) \leq \delta} \prod_{i=1}^k \frac{1}{d^{\nabla \D}(\star_{2i-1},\star_{2i})} \right ]  \frac{d\delta}{\delta^2}, \] 
%and the desired result then follows from the definitions of $G$ and $g$.
\end{proof}

\begin{lemma} \label{Concentretoi} Let $n\geq 2$, $m \geq 0$. Let $(S_i)_{1\leq i \leq n}$ be independent geometric random variables of mean $m$. Then,
\[ \E\left [\frac{1}{\sum_{i=1}^n (1+S_i)} \right ]\leq \frac{2e}{n(1+m)}. \] 
Also, for every $\e>0$,
\[ \E \left [\frac{\1_{\sum_{i=1}^n (1+S_i)\leq \e}}{\sum_{i=1}^n (1+S_i)} \right ]\leq \frac{2e}{n(1+m)} \min\left (1,\frac{e\e}{(1+m)n}\right ). \] 
\end{lemma}
\begin{proof} Note that $\sum_{i=1}^n (1+S_i)$ is the time needed to get $n$ success for Bernoulli trials that hold with probability $1/(1+m)$. Thus for every $x>0$,
\[ \proba \left (\sum_{i=1}^n (1+S_i)\leq x  \right )\leq \binom{\lfloor x\rfloor }{n}\frac{1}{(1+m)^n} \leq  \left (\frac{x}{1+m}\right )^n/n! \leq \left (\frac{ex}{(1+m)n}\right )^n. \]

It directly follows by integration by part that,
\begin{align*} \E\left [\frac{1}{\sum_{i=1}^n (1+S_i)}  \right ] & = \int_0^\infty \proba \left (\sum_{i=1}^n (1+S_i) \leq x \right ) x^{-2} dx\notag
\\ & \leq \int_0^{(1+m)n/e} \left (\frac{ex}{(1+m)n}\right )^n x^{-2} dx +\int_{(1+m)n/e}^{\infty} x^{-2}dx \notag
\\ & = \frac{e}{(1+m)n(n-1)}+\frac{e}{(1+m)n} \notag
\\ & \leq \frac{2e}{(1+m)n}.  \end{align*} %\label{0112.10h}

The second inequality is proved in a similar way.
%Similarly, for every $\e\geq 0$,
%\begin{align*} \E\left [\frac{\1_{\sum_{i=1}^n (1+S_i)\leq \e}}{\sum_{i=1}^n (1+S_i)}  \right ] & = \int_0^\infty \proba \left (\sum_{i=1}^n (1+S_i) \leq \min(x,\e) \right ) x^{-2} dx
%\\ & \leq \int_0^{\e} \left (\frac{ex}{(1+m)n}\right )^n x^{-2} dx +\int_{\e}^{\infty} \left (\frac{e\e}{(1+m)n}\right )^n x^{-2}dx
%\\ & = \e^{n-1} \left (\frac{e}{(1+m)n}\right )^n \left(\frac{1}{n-1} +1\right )
%\\ & \leq \frac{2e}{(1+m)n} \left (\frac{e\e}{(1+m)n}\right )^{n-1}. \end{align*}
%The last desired inequality then follows from \eqref{0112.10h} and $n\geq 1$.
%
%Thus by \eqref{0112.10h} and $n\geq 2$,
%\begin{equation} \E \left [\frac{\1_{\sum_{i=1}^n (1+S_i)\leq \e}}{\sum_{i=1}^n (1+S_i)} \right ]\leq \frac{2e}{n(1+m)} \min\left (1,\frac{e\e}{(1+m)n}\right ). \label{0112.10h10} \end{equation}
%
%Finally we have for every $x\geq 0$,
%\[ \min\left (1,x\e \right )=\e\int_0^{ \frac{1}{\e}} \1_{t \leq x} dt=\e\int_\e^{+\infty} \1_{1/x \leq \delta} \frac{d\delta}{\delta^2},\]
%and the desired result follows from \eqref{0112.10h10}.
%\end{proof}
\end{proof}
\subsection{Bias of $\P$-trees and ICRT}
Recall the definitions of section \ref{DkDef} and \ref{TkDef} of $(\square_i)_{1\leq i \leq k}$ and $\square_{\square,k}$. Recall that for every $x,m\in \R^+$,  $h_m :x \mapsto \1_{x\geq m}x$. %Note that for every $\e\in[0,\infty]$ and $\D\in \OmegaD$ with $N^\D\geq 2k$,
% $f^\D(\e)=\E[h_\e ((\s^\D/\n^\D)^k \square_{\square,k}(T^{\D})) ]$. 
%Similarly we let for every $\P\in \OmegaP$, $f^\P(\e):=\E [ h_\e ((\s^\P)^k \square_{\square,k}(T^{\P} ) )]$,
%and  for every $\Theta\in \OmegaT$, $f^\Theta(\e):=\E [h_\e(\square_{\square_k}(\textbf{Y}^\Theta, \textbf{Z}^\Theta)) ]$.
%\begin{lemma} 
%\begin{compactitem} We have the following result for every $c\in \N$:
%\item[a)] For every tree $T$, if $\{\star_a\}_{1\leq a \leq 2c}$ are leaves of $T$ and $d$ denote the graph distance on $T$ then 
%\[ \square_c(T)\leq d(\star_1,\star_2)+\sum_{b=2}^{2c} \min_{1\leq a\leq b}\left (d(\star_1,\star_b)+d(\star_a,\star_b)-d(\star_1,\star_a) \right ).\]
%\item[b)] 
%\end{compactitem}
%\end{lemma} 
%We have the following results:
\begin{lemma} \label{Followers} We have the following assertions:
\begin{compactitem}
\item[(a)] 
\[ \lim_{m\to \infty} \max_{\P\in \OmegaP}\E\left [h_{m}\left (\square_{\square,k}(T^{\P})/(\sigma^\P)^k\right )\right ]=0. \]
%For every $\delta>0$, there exists $\e>0$ such that for every $\P\in \OmegaP$, $f^\P(\e)\leq \delta$.
\item[b)] \[ \lim_{m\to \infty} \max_{\Theta \in \OmegaT}\E\left [h_{m}\left (\square_{\square,k}(\textbf Y^\Theta, \textbf Z^\Theta)\right )\right ]=0. \]
%For every $\delta>0$, there exists $\e>0$ such that for every $\Theta\in \OmegaT$, $f^\Theta(\e)\leq \delta$. 
\end{compactitem} 
\end{lemma}
\begin{proof} We focus only on (a) as (b) can be proved in the exact same way. Fix $\P\in \OmegaP$. Let $(\Dn)_{n\in \N}\in \OmegaD^\N$ such that $\Dn \ply \P$ (see the start of Section \ref{MAINsection} or \cite{Uniform} Section 8.1 for existence). %For every $1\leq i \leq 2k$ let $W_i=\star_i$. For every $i>2k$ let $W_{i}$ be any vertex in $\{V_i\}_{i:p_i^\P>0}$.  
By \cite{Uniform} Theorem 5, we have the following weak convergence,
\[ (d^\Dn(\star_{i},\star_{j}))_{1\leq i,j \leq 2k}\limit^{(d)} (d^\P(\star_i,\star_j))_{1\leq i,j\leq 2k}. \]
Then by Lemma \ref{CHIANTa} (see also \cite{MST} Corollary 6.6), $\square_{\square_k}(T^{\Dn})$ converges weakly toward $\square_{\square_k}(T^{\P})$ as $n\to \infty$.
Furthermore, by Fubini's Theorem, 
\[(\lambda^\Dn)^2=(\sigma^\Dn/s^\Dn)^2= \sum_{i=1}^\infty \frac{(d^\Dn_i)(d^\Dn_i-1)}{(\n^\Dn)^2}\limit \sum_{i=1}^\infty p_i^2=(\sigma^\P)^2. \]
Therefore, for every $m\geq 0$, % we have by dominated convergence and Proposition \ref{MainMainMain}, for every $m\geq 0$,
\begin{equation} \limsup \E[h_m ( \square_{\square,k}(T^{\Dn})/(\lambda^\Dn)^k) ]\geq \E [ h_{m+1} ( \square_{\square,k}(T^{\P} ) /(\s^\P)^k)].\label{21h/4/12} \end{equation}
Finally, Proposition \ref{MainMainMain} concludes the proof.
\end{proof}
%Along the way, by \cite{} we also prove that the map $\Lambda\in \Omega \mapsto f^\Lambda(\e)$ is continuous. In other words:
%\begin{lemma}% \label{MainMainTheta} Let $\e>0$. Let $(\D_n)_{n\in \N}, (\P_n)_{n\in \N}, (\Theta_n)_{n\in \N}$ be some sequences in $\OmegaD,\OmegaP, \OmegaT$.
%\begin{compactitem}
%\item[(a)] If $\Dn\ply \P\in \OmegaP$, then $f^{\Dn}(\e)\to f^\P(\e)$.
%\item[(b)] If $\Dn\ply \Theta\in \OmegaT$, then $f^{\Dn}(\e)\to f^\Theta(\e)$.
%\item[(c)] If $\Pn\ply \Theta\in \OmegaT$, then $f^{\Pn}(\e)\to f^\Theta(\e)$.
%\item[(d)] If $\Theta_n\ply \Theta\in \OmegaT$, then $f^{\Pn}(\e)\to f^\Theta(\e)$.
%\end{compactitem}
%\end{lemma} 
%\begin{proof} For (a) see \eqref{21h/4/12}. The proofs of (b), (c), (d) are similar.
%\end{proof}
%\begin{proof} The proof is similar.%Fix $\Theta\in \OmegaT$. Let $(\Dn)_{n\in \N}$ be a sequence in $\OmegaD$ such that $\Dn \ply \Theta$. By \cite{Uniform} Theorem 5, we have the following weak convergence, 
%\[ ((\s^\Dn/\n^\Dn)d^\Dn(\star_{i},\star_{j}))_{1\leq i,j \leq 2k}\limit^{(d)} (d^\P(Y_i,Y_j))_{1\leq i,j\leq 2k}. \]
%Again, it follows that weakly, (cf Corollary 6.6 (a) of \cite{MST})
%\[ (\s^\Dn/\n^\Dn)^k\square_{\square,k}(T^\Dn) \limit^{(d)} \square_{\square,k}(T^\Theta). \]
%\end{proof} 
\section{Proof of the main theorems} \label{May the proof section be removed?}
 Theorems \ref{GP} and \ref{THM2} directly follows from three thing: the trees converges, the operation of gluing leaves is a continuous application, and the bias converge. In this section, we precise the proofs. 
 \subsection{Proof of Theorem \ref{GP}} %To simplify the notations we write in this section $[G^{\Dn,k}]$  instead of $( G^{\Dn,k},\s^\Dn/\n^\Dn d^{\Dn,k},\M^\Dn)$, $[G^{\Dn,k}]$  instead of $( G^{\Pn,k},\s^\Pn d^{\Pn,k},\M^\Pn)$, and $[G^{\Tn,k}]$  instead of $(G^{\Tn,k},d^{\Tn,k},\p^{\Tn,k}) $. 
\begin{proof}[Proof of Theorem \ref{GP} (a)] Let $(\Dn)_{n\in \N}\in \OmegaD^\N$ and  $\P=(p_i)_{i\in \N\cup \{\infty\}}\in \OmegaP$ such that $\Dn\ply \P$. Let $a\in \N$ such that $p_a>0$. For all $1\leq i \leq a$ let $W_i=V_i$. For all $1\leq i \leq 2k$, let $W_{a+i}:= \star_i$. 
By \cite{Uniform} Theorem 5, it is easy to check that we have the following joint convergence,
\begin{equation} (d^\Dn(W_i,W_j))_{1\leq i,j\leq a+2k}\limit^{(d)} (d^\P(W_i,W_j))_{1\leq i,j\leq a+2k}, \label{10/12/8h} \end{equation}
writing $d^\Dn$ for the graph distance on $T^\Dn$, and $d^\P$ for the graph distance on $T^\P$.

Then by Kolmogorov representation theorem, we may assume that \eqref{10/12/8h} holds a.s.  Furthermore, since we work with discrete trees, note that a.s. for every $n$ large enough equality holds in \eqref{10/12/8h}. Hence, by Lemma \ref{CHIANTa} a.s. for every $n$ large enough $\square_{\square,k}(T^\D)=\square_{\square,k}(T^\P)$. Thus, by dominated convergence, for any continuous bounded function $f:\R^{(a+2k)^2}\to \R^+$,
\[ \frac{\E[f(d^\Dn(W_i,W_j))_{1\leq i,j\leq a+2k})\square_{\square,k}(T^\Dn)]}{\E[\square_{\square,k}(T^\Dn)]}\limit \frac{\E[f(d^\P(W_i,W_j))_{1\leq i,j\leq a+2k})\square_{\square,k}(T^\P)]}{\E[\square_{\square,k}(T^\P)]}. \]
Therefore, writing $\bar d^{\Dn,k}$ for the graph distance on $T^{\Dn,k}$ and $\bar d^{\P,k}$ for the graph distance on $T^{\P,k}$,
\begin{equation} (\bar d^{\Dn,k}(W_i,W_j))_{1\leq i,j\leq a+2k}\limit^{(d)} (\bar d^{\P,k}(W_i,W_j))_{1\leq i,j\leq a+2k}. \label{Tardif} \end{equation}

Finally by gluing $(\star_1,\star_2),\dots, (\star_{2k-1},\star_{2k})$, which is a continuous map for the matrix distance, 
\begin{equation*} (d^{\Dn,k}(V_i,V_j))_{1\leq i,j\leq a}\limit^{(d)} (d^{\P,k}(V_i,V_j))_{1\leq i,j\leq a}.\end{equation*}
And Theorem \ref{GP} (a) follows from Lemma \ref{equivGP2}.
\end{proof}
\begin{proof}[Proof of Theorem \ref{GP} (b)] Let $(\Dn)_{n\in \N}\in \OmegaD^\N$ such that $\Dn\ply \Theta\in \OmegaT$. For every $n\in \N$ let $\M^{\Dn,k}$ be a probability measure on $\V^{\Dn,k}$ such that $\M^{\Dn,k}\to 0$.  For every $n\in \N$ and $1\leq i \leq 2k$, let $W^\Dn_{i}:= \star_i$. Also, let $(W_i^\Dn)_{i>2k}$ be a family of independent random variables with law $\M^{\Dn,k}$. Fix $a>2k$. By \cite{Uniform} Theorem 6 (b) and Lemma 14, we have %writing for $\D\in \OmegaD$, $\lambda^\D:=\s^\D/\n^\D$,
\begin{equation} \left (\lambda^\Dn d^\Dn(W^\Dn_i,W^\Dn_j) \right)_{1\leq i,j\leq a}\limit^{(d)} (d^\Theta(Y^\Theta_i,Y^\Theta_j))_{1\leq i,j\leq a}. \label{10/12/8hX} \end{equation}

Then by Kolmogorov representation theorem we may assume that \eqref{10/12/8hX} holds almost surely. Hence, by Lemma \ref{CHIANTa} a.s. $ \square_{\square,k}(T^{\Dn})/(\lambda^\Dn)^k\to \square_{\square,k}(Y^\Theta,Z^\Theta)$ as $n\to \infty$.  Thus, by Proposition \ref{MainMainMain} and dominated convergence, we have for all continuous bounded function $f:\R^{a^2}\to \R$, 
\[ \frac{\E[f((\lambda^\Dn d^\Dn(W^\Dn_i,W^\Dn_j))_{1\leq i,j\leq a})\square_{\square,k}(T^\Dn)]}{\E[\square_{\square,k}(T^\Dn)]}\limit \frac{\E[f((d^\Theta(Y^\Theta_i,Y^\Theta_j))_{1\leq i,j\leq a})\square_{\square,k}(Y^\Theta,Z^\Theta)]}{\E[\square_{\square,k}(Y^\Theta,Z^\Theta)]}. \]
Therefore, %writing $\bar d^{\Theta,k}$ for the distance on $\T^{\Theta,k}$,
\begin{equation*} (\lambda^\Dn\bar d^{\Dn,k}(W^\Dn_i,W^\Dn_j))_{1\leq i,j\leq a}\limit^{(d)} (\bar d^{\Theta,k}(Y^\Theta_i,Y^\Theta_j))_{1\leq i,j\leq a}.\end{equation*} 

Finally by gluing the $k$ first pair of vertices, which is a continuous map for the matrix distance,
\begin{equation*} (\lambda^\Dn d^{\Dn,k}(W^\Dn_i,W^\Dn_j))_{2k+1\leq i,j\leq a}\limit^{(d)} (d^{\Theta,k}(Y^\Theta_i,Y^\Theta_j))_{2k+1\leq i,j\leq a}.\end{equation*} 
And Theorem \ref{GP} (a) follows from Lemma \ref{equivGP2}.
\end{proof}
%\begin{proof}[Proof of Theorem \ref{GP} (c)] Since $\K_{\GP}$ is a Polish space, the space of probability measure on $\K_{\GP}$ with the weak convergence topology is also a Polish space. Let $d_{\WGP}$ be a distance on this space. 
%
% Let $(\Pn)_{n\in \N}\in \OmegaP^\N$ and $\Theta\in \OmegaT$ such that $\Pn\ply \Theta$. Let $(\M^\Pn)_{n\in \N}$ a sequence of probability measure on $(\V^\Pn)_{n\in \N}$ such that $\M^\Pn\to 0$.  By density of $\Dn$ on $(\Omega,\ply)$ there exists $(\Dn)_{n\in \N}$ a sequence of $\OmegaD$, and $(\M_n)_{n\in \N}$ a sequence of probability measure on $(\V^\Dn)_{n\in \N}$ (see \cite{Uniform} Section 8.1 for details) such that $\Dn\ply \Theta$, and such that as $n\to \infty$, 
% \[ d_{\WGP}([G^{\Dn,k}],[G^{\Pn,k}])\to 0. \]
% 
% Furthermore, by Theorem \ref{GP} (b), we have as $n\to \infty$,
% \[ d_{\WGP}([G^{\Dn,k}],[G^{\Tn,k}])\to 0.\]
%  Therefore, as $n\to \infty$, 
%  \[ d_{\WGP}([G^{\Pn,k}],[G^{\Tn,k}])\to 0. \qedhere\]
%\end{proof}
%\begin{remark} We could have use the same argument to prove the convergence of the matrix distance. And then deduce the result from Lemma \ref{equivGP2}.
%\end{remark}
%\begin{proof}[Proof of Theorem \ref{GP} (d)] The proof is exactly the same as in (c).
\begin{proof}[Proof of Theorem \ref{GP} (c,d)] Since $\K_{\GP}$ is a Polish space, and $\Omega_{\D}$ is dense on $(\Omega,\ply)$, the results directly follows from Theorem \ref{GP} (a,b) (see \cite{Uniform} Section 8.1 for details). Also, they can be proved similarly.
\end{proof}
 \subsection{Proof of Theorem \ref{THM2}}
%Recall that for every tree, and every $\R$-tree, $T$ and $v_1,\dots, v_n\in T$, we write $T(\{v_i\}_{1\leq i\leq n})$ for the subtree spanned by $v_1, \dots, v_n$.
 %By Assumption \ref{Hypo3}, \ref{Hypo3P}, \ref{Hypo3T} it is enough to prove the proposition:
% \begin{proposition} \label{LOLLOLOL} The following convergences hold weakly for the GH-topology. 
%\begin{compactitem}
%\item[(a)] If $\Dn\ply \Theta$, $\M^\Dn \to 0$, then
%\[ \left ( G^{\Dn,k},(\s^\Dn/\n^\Dn) d^{\Dn,k} \right) \limit^{\WGH} (G^{\Theta,k},d^{\Theta,k}).  \]
%\item[(b)]  If $\Pn \ply \Theta$, $\M^\Pn \to 0$, then
%\[ \left ( G^{\Pn,k},\s^\Pn d^{\Pn,k} \right) \limit^{\WGH} (G^{\Theta,k},d^{\Theta,k}).  \] 
%\item[(c)] If $\Tn \ply \Theta$, then
%\[(G^{\Tn,k},d^{\Tn,k}) \limit^{\WGH} (G^{\Theta,k},d^{\Theta,k}).  \]
%\end{compactitem}
%\end{proposition}
\begin{proof}[Proof of Theorem \ref{THM2} (a)] Let $(\Dn)_{n\in \N}\in\OmegaD^\N$ such that $\Dn\ply \Theta\in \OmegaT$. By \cite{Uniform} Theorem 6 (b),
\begin{equation*} (\lambda^\Dn d^\Dn(\star_i,\star_j))_{i,j\in \N}\limit^{(d)} (d^\Theta(Y^\Theta_i,Y^\Theta_j))_{1\leq i,j\in \N}. \end{equation*}
Thus, by Lemma \ref{reconstructTHM} for every $a\in \N$, we have for the a-pointed GH topology (see Section \ref{PointedGH}),
\begin{equation*} (T^\Dn(\{\star_i\}_{1\leq i \leq a}),\lambda^\Dn d^\Dn,\{\star_i\}_{1\leq i \leq a}) \limit^{\WGH^a} (T^\Theta(\{Y_i^\Theta\}_{1\leq i \leq a}),d^\Theta,\{Y_i^\Theta\}_{1\leq i \leq a}). \end{equation*}
Therefore, by Assumption \ref{Hypo3}, we have for the $2k$-pointed GH topology,
\begin{equation} (T^\Dn,\lambda^\Dn d^\Dn,\{\star_i\}_{1\leq i \leq 2k}) \limit^{\WGH^{2k}} (T^\Theta,d^\Theta,\{Y_i^\Theta\}_{1\leq i \leq 2k}). \label{10/12/16h} \end{equation}

Then, by Skorohod representation theorem we may assume that the above convergence holds almost surely. Thus by Lemma \ref{CHIANTa} a.s. $\square_{\square,k}(T^\Dn)\to \square_{\square,k}(\textbf Y^\Theta, \textbf Z^\Theta)$.
Then for every continuous bounded function $f$ on $\K^{2k}_{\GH}$ we have by Proposition \ref{MainMainMain} and dominated convergence,
\[ \frac{\E[f(T^\Dn,\lambda^\Dn d^\Dn,\{\star_i\}_{1\leq i \leq 2k})  \square_{\square,k}(T^\Dn)]}{\E[\square_{\square,k}(T^\Dn)]}\limit \frac{\E[f(T^\Theta,d^\Theta,\{Y_i^\Theta\}_{1\leq i \leq 2k}) \square_{\square,k}(\textbf Y^\Theta, \textbf Z^\Theta)]}{\E[\square_{\square,k}(\textbf Y^\Theta,\textbf Z^\Theta)]}. \]
Therefore,
\[ (T^{\Dn,k},\lambda^\Dn \bar d^{\Dn,k},\{\star_i\}_{1\leq i \leq 2k})\limit^{\WGH^{2k}} (T^{\Theta,k}, \bar d^{\Theta,k},\{Y_i^{\Theta,k}\}_{1\leq i \leq 2k}). \]

Finally since the gluing of $k$ pair of point is a continuous operation for the $2k$-pointed GH topology the desired result follows.
\end{proof}
\begin{proof}[Proof of Theorem \ref{THM2} (b,c)] The results can be proved in the exact same way. %We omit the details.
\end{proof}
\section{Configuration model and multiplicative graphs} \label{ALTEsection}
The main objective of this section is to explain the connections between the configuration model and multiplicative graphs, and between those models and $(\D,k)$-graphs and $(\P,k)$-graphs. %Since we deal with another type of model some notations differ from the rest of the paper.
%
%This section is slightly disconnected with the rest of the paper and study configuration models and multiplicative graphs. 
%\subsection{Configuration model and multiplicative graphs: definition and link} 
\subsection{Definitions}
For every multigraph $G$ on $\{V_i\}_{i\in \N}$ and $i,j\in \N$ let $\#_{i,j}(G)$ be the number of edges $\{V_i,V_j\}$ in $G$. So that a multigraph on $\{V_i\}_{i\in \N}$ may be seen as a matrix.

We call a function $f:I\mapsto I$ a matching if $f\circ f=\Id$ and for every $x\in I$, $f(x)\neq x$. Let $\Omega_{\CM}$ be the set of decreasing sequence $(d_1,\dots d_\n)$ in $\{0\}\cup \N$ such that $\sum_{i=1}^\n d_i$ is even. 
\begin{algorithm} \label{ConfigAlgo} Construction of the configuration model from $\D=(d_1,\dots, d_\n)\in\Omega_{\CM}$:
\begin{compactitem}
\item[-] Let $f=(f_1,f_2)$ be a uniform matching of $\{(i,j) \}_{1\leq i \leq \n, 1\leq j\leq d_i}$.
\item[-] The configuration model is the random multigraph $\CM^\D$ with vertices $(V_i)_{1\leq i \leq \n}$ and such that for every $1\leq i \leq \n$,  $\#_{i,i}(\CM^\D):=\frac{1}{2} \sum_{a=1}^{d_i} \1_{f_1(i,a)=i}$ and for $1\leq i \neq j\leq \n$, 
\[ \#_{i,j}(\CM^\D): =\sum_{a=1}^{d_i} \1_{f_1(i,a)=j}=\sum_{a=1}^{d_j} \1_{f_1(j,a)=i}.\]
\end{compactitem}
\end{algorithm}
Let $\Omega_{\MG}$ be the set of sequence $(\lambda,p_1,\dots, p_\n)$ in $\R^{+*}$ with $p_1\geq \dots\geq p_\n$. 
\begin{algorithm} Construction of the multiplicative graph from $\P=(\lambda,p_1,\dots, p_\n)\in\Omega_{\MG}$:
\begin{compactitem}
\item[-] Let $(X^\P_{i,j})_{1\leq i\neq j \leq \n}$ be independent Bernoulli random variables with mean $1-e^{-\lambda p_i p_j}$.
\item[-] The multiplicative graph is the random graph $\MG^\P$ with vertices $(V_1,\dots, V_\n)$ and with edges $\{1\leq i,j \leq \n: X_{i,j}=1 \}$.
\end{compactitem}
\end{algorithm}
Next, we introduce multiplicative multigraphs, which are augmented multiplicative graphs.
\begin{algorithm} Construction of the multiplicative multigraph from $\P=(\lambda,p_1,\dots, p_\n)\in\Omega_{\MG}$:
\begin{compactitem}
\item[-] Let $(N^\P_{i,j})_{1\leq i,j \leq \n}$ be independent Poisson random variables, such that for every $1\leq i \leq \n$, $N^\P_{i,i}$ have mean $\lambda p_i^2/2$ and for every $1\leq i\neq j \leq \n$, $N^\P_{i,j}$ have mean $\lambda p_ip_j$.
\item[-] The multiplicative multigraph is the random multigraph $\MG^{\P+}$ with vertices $(V_i)_{1\leq i \leq \n}$ and such that for every $1\leq i,j\leq \n$,  $\#_{i,j}(\MG^{\P+}):=N^\P_{i,j}$.
\end{compactitem}
\end{algorithm} 
\begin{lemma} There exists a coupling such that $\MG^\P$ is the graph obtained from $\MG^{\P+}$ by removing all its multi-edge. That is, for every $i\neq j$, $\{i,j\}$ is an edge of $\MG^\P$ iff $\#_{i,j}(\MG^{\P+})\geq 1$.
\end{lemma}
\begin{proof} It is easy to check that there exists a coupling such that a.s. for every $1\leq i \neq j \leq \n$ $X^\P_{i,j}=0$ iff $N^\P_{i,j}=0$. The result follows.
\end{proof}
\subsection{Multiplicative multigraphs as local limit of the configuration model}
\begin{lemma} \label{CM=>MG}Let $\P=(\lambda,p_1,\dots, p_\n)\in\Omega_{\MG}$. For $n\in \N$, let $\D^n=(d_i^n)_{1\leq i \leq \n^n}\in \Omega_{\CM}$. If $\n^\Dn\to \infty$, and for every $1\leq i \leq \n$, $d^n_i\sim  \sqrt{ \n^n \lambda} p_i$, and for every $n\in \N$, $i> \n$, $d_i^n=1$. Then,
\[\left  (\#_{i,j}(\CM^{\Dn}) \right )_{1\leq i,j\leq \n}\limit^{(d)}\left  (\#_{i,j}(\MG^{\P+})\right )_{1\leq i,j\leq \n}.  \]
\end{lemma}
\begin{remark} From this result, one may see the LIFO-queue algorithm of Broutin, Duquesne, Wang \cite{P-graph-1,P-graph-2} as a limit of a recursive construction, based on a DFS exploration, of a uniform matching.
\end{remark}
\begin{proof}  Let $(\D^n)_{n\in \N}$ and $\P$ be as in the statement. For $n\in \N$, let 
 $f^n=(f^n_1,f^n_2)$ a uniform matching of $\{(i,j) \}_{1\leq i \leq \n^n, 1\leq j\leq d^n_i}$. We may assume that $\CM^{\Dn}$ is constructed from $f^n$ by Algorithm \ref{ConfigAlgo}. The  main idea is that for $n$ large enough $\{f_1(i,j)\}_{1\leq i \leq \n^n, 1\leq j\leq d^n_i}$ are mostly independent. Since Poisson random variables appears as the limits of Bernoulli trials this explain the  convergence. From there, there are many standard ways to justify the convergence. 

Below we briefly present a method based on random point process. We let the reader refer to Kallenberg \cite{Kallenberg} Section 4 for more details on convergence of point process. Let $\nu^n$ be the random measure on $\K:=\{(i,j)\}_{1\leq i\leq j \leq \n}\times \R^2$ defined by
\[ \nu^n:=\sum_{1\leq i <j\leq \n}\sum_{\substack{1\leq a \leq d_i\\1\leq b \leq d_j } } \1_{f(i,a)=(j,b)}\delta_{(i,j,a/\sqrt{ \n^n},b/\sqrt{ \n^n})} +\sum_{1\leq i\leq n}\sum_{1\leq a<b\leq d_i } \1_{f(i,a)=(i,b)} \delta_{(i,i,a/\sqrt{ \n^n},b/\sqrt{ \n^n})} . \]
It is enough to prove that $\{\nu^n\}_{n\in \N}$ converges vaguely toward a Poisson point process of rate 
\begin{equation}  d\nu:=\sum_{1\leq i\leq j\leq \n} \1_{0\leq x\leq \lambda p_i}\1_{0\leq y\leq \lambda p_j} \delta_{i,j}dx dy+\sum_{1\leq i\leq n}  \1_{0\leq x\leq y \leq \lambda p_i}\delta_{i,i}dx dy. \label{11/12/13h} \end{equation}
Indeed, provided this convergence, the desired result directly follows by integration over $dxdy$.

To this end, first note that for every $n\in \N$, writing $m_n:=\sum_{i=1}^{\n^n} d_i^n$, 
\begin{align*} \E[\nu^n(\K)] & =\sum_{1\leq i <j\leq \n}\sum_{\substack{1\leq a \leq d_i\\1\leq b \leq d_j } } \proba(f(i,a)=(j,b))+\sum_{1\leq i\leq n}\sum_{1\leq a<b\leq d_i } \proba(f(i,a)=(i,b)).
\\ & =  \sum_{1\leq i <j\leq \n} \frac{d_i d_j}{m^n}+\sum_{1\leq i\leq n}\sum_{1\leq a<b\leq d_i } \frac{d_i^2/2}{m^n}
\\ & \to \sum_{1\leq i <j\leq \n} \lambda p_i p_j+\sum_{1\leq i\leq n}\sum_{1\leq a<b\leq d_i } \lambda p_i^2/2
 \end{align*}
 where the last inequality comes from the assumptions of the lemma on $(\D_n)_{n\in \N}$. Thus, $\{\nu^n\}_{n\in \N}$ is tight for the vague topology.
 Let $\nu$ be a sub-sequential limit of $\{\nu^n\}_{n\in \N}$.
 
By a similar computation, for every $1\leq i<j\leq \n$, and $0\leq a\leq a' \leq \lambda p_i$, $0\leq b\leq b' \leq \lambda p_j$,
 \[\E[\nu(\{i,j\} \times [a,a']\times [b,b'])]=\lim_{n\to\infty} \E[\nu^n(\{i,j\}\times [a,a']\times [b,b'] )]= \lambda (a'-a)(b'-b). \]
 And for every $1\leq i \leq \n$, $0\leq a\leq a'\leq b\leq b'\leq p_i$.
  \[\E[\nu(\{i,i\} \times [a,a']\times [b,b'])]=\lim_{n\to\infty} \E[\nu^n(\{i,i\}\times [a,a']\times [b,b'] )]= \lambda (a'-a)(b'-b). \]
  
Next, we prove that $\nu$ satisfies the independency criterium. Beforehand let us introduce some notations. Let $\cov(\cdot,\cdot)$  be the covariance of two random variables. Let
 \begin{align*}S := \left \{(i,j,a,b)\in \N^4: 1\leq i<j\leq \n, \substack{ 1\leq a \leq d_i \\ 1\leq b\leq d_j } \right \} \cup \{(i,i,a,b)\in \N^4: 1\leq i\leq \n, 1\leq a<b \leq d_i \}. \end{align*}
For every $K_1, K_2\subset \K$ disjoint compact set, for every $n\in \N$, $\cov(\nu^n(K_1),\nu^n(K_2))$ equal
 \[ \sum_{\substack{(i,j,a,b)\in S\\(i',j',a',b')\in S }}\1_{\left (i,j,\frac{a}{\sqrt{\n^n}},\frac{b}{\sqrt{\n^n}} \right )\in K_1}\1_{\left (i',j',\frac{a'}{\sqrt{\n^n}},\frac{b'}{\sqrt{\n^n}}\right )\in K_2} \cov(\1_{f(i,a)=(j,b)},\1_{f(i',a')=(j',b')}).  \]
 Then, by distinguishing whether it is possible to have both $f(i,a)=(j,b)$ and $f(i',a')=(j',b')$, note that in the last sum there are $O(\#S)$ terms that are equal to $0-1/(m^n)^2$, $O((\#S)^2)$ terms that are equal to $1/(m^n)(m^n-2)-(1/m^n)^2=O(1/(m^n)^3)$, and the others that are null. Therefore,
 \[ |\cov(\nu^n(K_1),\nu^n(K_2))| =O(\#S)O(1/(m^n)^2)+O((\#S)^2)O(1/(m^n)^3)=O(1/m^n)\to 0.\]
 Since the last convergence hold for every disjoint compact $K_1,K_2\subset \K$, we have that for every disjoint compact $K'_1,K'_2\subset \K$, $\cov(\nu(K'_1),\nu(K'_2))=0$. 
 
 Finally, to prove that $\nu$ is a Poisson point process of rate \eqref{11/12/13h} it is enough to check that a.s. for every $x\in \K$, $\nu(x)\in \{0,1\}$. To this end, one may adapt the previous argument to show that there exists $C>0$, such that for every $x\in \K$, $\e>0$, writing $B(x,\e)$ for the closed ball centered at $x$ of radius $\e$ for $\| \|_{\infty}$, if $B(x,\e)$ does not intersect $\{(i,i,1/2,1/2)\}_{1\leq  i\leq \n}$ then
 \[ \E[\nu(B(x,\e))(\nu(B(x,\e))-1)]\leq C\e^2. \]
This implies the desired property, and so concludes the proof.
 %This concludes the proof.
%And then desired result directly follows from this convergence by integration over the two last coordinates.
\end{proof}
\subsection{Connections with $(\D,k)$-graphs and $(\P,k)$-graphs}
Recall that for every multigraph $G$ on $\{V_i\}_{i\in \N}$, $\circ(G):=\prod_{i\in \N} 2^{\#_{i,i}(G)}\prod_{i,j\in \N}\#_{i,j}(G)!$.
\begin{lemma} \label{Connections}
\begin{compactitem} Let $k\in \N$ we have the following assertions:
\item[(a)] Let $\D=(d_1,\dots, d_\n)\in  \Omega_{\CM}$ such that $\sum_{i=1}^\n d_i=2\n+k-2$. Then $\CM^\D$ biased by $\circ(\CM^\D)$ and conditioned at being connected is a $((d_1-1,\dots, d_\n-1),k)$-graph.
\item[(b)] Let $\mathcal W=(\lambda,w_1,\dots, w_\n)\in \Omega_{\MG}$. For every $1\leq i \leq \n$, let $p_i:=w_i/\sum_{j=1}^\n w_j$. Let  $\P=(p_1, \dots, p_\n,0,0,\dots)$. Then $\MG^{\mathcal W+}$ biased by $\circ(\MG^{\mathcal W+})$ and conditioned at being connected and having surplus $k$ is a $(\P,k)$-graph. 
\end{compactitem}
\end{lemma}
\begin{remark} The bias is not really important as typically those graphs are studied in a regime where with high probability the multigraph is a graph. Also removing this bias only remove the term $\circ(\G_{(\star_i)_{1\leq i \leq 2k}}(T))$ in Section \ref{DkDef} which does not change our proofs.
\end{remark}
\begin{proof} (a) is a classic and is easy to obtain from a quick enumeration. So we focus on (b). The main idea is that, on the one hand multiplicative multigraph are limits of the configuration model, and on the other hand $(\P,k)$-graph are limits of $(\D,k)$-graph. Thus by identification, (b) follows. Let us detail:

Fix $k,\mathcal W ,\P \in \Omega_{\MG}$ as in (b). Let $(\D^n)_{n\in \N}$ be a sequence of $\Omega_{\CM}$ as in Lemma \ref{CM=>MG}.
Then write $\CM^{\W,k}$ for the random multigraph $\MG^{\mathcal W+}$ biased by $\circ(\MG^{\mathcal W})$ and conditioned at being connected and having surplus $k$. Also, write for $n\in \N$, $\CM^{\D^n, k}$ for the random multigraph $\CM^{\D^n}$ biased by $\circ(\CM^{\D^n})$ conditioned on the fact that the subgraph of $\CM^{\D^n}$ on  $(V_i)_{1\leq i \leq \n}$ is connected and have surplus $k$. By Lemma \ref{CM=>MG}, we have,
\begin{equation}  (\#_{i,j}(\CM^{\D^n,k})  )_{1\leq i,j\leq \n}\limit^{(d)}  (\#_{i,j}(\MG^{\mathcal W,k}))_{1\leq i,j\leq \n}. \label{9/12/18h} \end{equation}

Then, for every $n\in \N$ let $S^n+2\n$ be the number of vertices that are in the connected component of $(V_i)_{1\leq i \leq \n}$ in $\CM^{\D^n}$. Then let $\D^{n-}:=(d_1^n,\dots, d_\n^n,1,\dots, 1 )$ with $S^n$ number 1 at the end. It is well known that for every $n\in \N$, conditioned on $S^n$, $\CM^{\D^n,k}$ have the same law as $\CM^{\D^{n-}}$ (where the vertices outside $(V_i)_{1\leq i \leq \n}$ in $\CM^{\D^n}$ have been relabeled). More precisely,
\[   (\#_{i,j}(\CM^{\D^n,k}) )_{1\leq i,j\leq \n} \egale^{(d)}  (\#_{i,j}(\CM^{\D^{n-}})  )_{1\leq i,j\leq \n}. \]
Therefore, it directly follows from \eqref{9/12/18h}, that if for $n\in \N$, $\CM^{\D^{n-},k}$ be the random multigraph $\CM^{\D^{n-}}$ biased by $\circ(\CM^{\D^{n-}})$ and conditioned at being connected, then
\begin{equation}  (\#_{i,j}(\CM^{\D^{n-,k}})  )_{1\leq i,j\leq \n}\limit^{(d)} (\#_{i,j}(\MG^{\mathcal W,k}))_{1\leq i,j\leq \n}. \label{9/12/19h} \end{equation}

Next let for $n\in \N$, $\D_n\in \OmegaD$ be the sequence $(d_1^n-1,\dots, d_\n^n-1,0,\dots, 0,0,\dots, 0)$ where we added $S^n+2k$ numbers $0$ at the end. We have by (a) for every $n\in \N$, 
\[   (\#_{i,j}(G^{\Dn,k}) )_{1\leq i,j\leq \n}  \egale^{(d)}   (\#_{i,j}(\CM^{\D^{n-},k})  )_{1\leq i,j\leq \n}\]
Therefore by \eqref{9/12/19h},
\begin{equation}     (\#_{i,j}(G^{\Dn,k})  )_{1\leq i,j\leq \n} \limit^{(d)}  (\#_{i,j}(\MG^{\mathcal W,k}) )_{1\leq i,j\leq \n}.  \label{9/12/21hb} \end{equation}

Finally note that $\Dn\ply \P$. So, by \eqref{Tardif} and Lemma \ref{reconstructTHM}, as $n\to \infty$ the subtree of $T^{\D_n,k}$ spanned by $\{V_i\}_{1\leq i \leq \n}\cup\{\star_i\}_{1\leq i \leq 2k}$ converges weakly toward the subtree of $T^{\P_n,k}$ spanned by the same vertices. Therefore, we have by gluing $(\star_1,\star_2),\dots, (\star_{2k-1},\star_{2k})$, then counting the edges,
\begin{equation*}     (\#_{i,j}(G^{\Dn,k})  )_{1\leq i,j\leq \n} \limit^{(d)}  (\#_{i,j}(G^{\P,k}) )_{1\leq i,j\leq \n}.\end{equation*}
And \eqref{9/12/21hb} concludes the proof.
\end{proof}
To conclude the section let us  compute the law of $(\P,k)$-graph.
\begin{lemma} \label{BIGFINAL} Let $k\in \N$. Let $(p_1,\dots, p_\n,0,0\dots)\in \OmegaP$. We have for every connected multigraph $G$ on $\{V_i\}_{1\leq i \leq \n}$ with surplus $k$, writing $\alpha$ for proportional,
\[ \proba(G^{\P,k}=G)\alpha \prod_{1\leq i,j \leq \n} (p_i p_j)^{\#_{i,j}(G)}. \]
\end{lemma} 
\begin{proof} Keep the notations of Lemma \ref{Connections} (b). By definition of $\Omega_{\MG}^{\mathcal W+}$, we have %for every connected multigraph $G$ on $\{V_i\}_{1\leq i\leq \n}$ with surplus $k$,
\[ \proba(\Omega_{\MG}^{\mathcal W+}=G)=\prod_{1\leq i<j\leq \n} \frac{(\lambda p_i p_j)^{\#_{i,j}(G)}e^{-\lambda p_i p_j}}{\#_{i,j}(G)!}\prod_{1\leq i \leq \n} \frac{(\lambda p_i^2/2)^{\#_{i,i}(G)}e^{-\lambda p_i^2/2}}{\#_{i,i}(G)!}.\]
So the result follows from Lemma \ref{Connections} (b).
%Thus, we have writing $\Omega_{\MG}^{\mathcal W,k}$ for $\MG^{\mathcal W+}$ biased by $2^{-\sum_{i=1}^\n \#_{i,i}(\MG^{\mathcal W})}$ and conditioned at being connected and having surplus $k$, and writing $\alpha$ for proportional,
%\begin{equation*} \proba(\Omega_{\MG}^{\mathcal W,k}=G)\alpha \prod_{1\leq i,j \leq \n} (p_i p_j)^{\#_{i,j}(G)}. \qedhere \end{equation*}
\end{proof}
\begin{remark} $\bullet$  When $k=0$ the result is well known and is a classical definition for $\P$-trees. \\
$\bullet$ When the weight of the edges is not multiplicative, one can still construct similar multigraphs. Moreover, Lemma \ref{BIGFINAL} is still true in this case. For $k=0$, this relates those models with the general spanning trees constructed by Aldous--Br\"oder algorithm \cite{AaldousBroder,AldousBbroder}.%We are however still lacking a configuration model corresponding to those graphs.
\end{remark}

%\bibliography{../Biblio_ICRT}

\begin{thebibliography}{10}

\bibitem{GHP}
R.~Abraham, J.-F. Delmas, and P.~Hoscheit.
\newblock A note on gromov-hausdorff-prokhorov distance between (locally)
  compact measure spaces.
\newblock {\em Electron. J. Probab.}, 18(14, 21.), 2013.

\bibitem{ABG}
L.~Addario-Berry, N.~Broutin, and C.~Goldschmidt.
\newblock The continuum limit of critical random graphs.
\newblock {\em Probab. Theory Relat. Fields}, 152(3-4):367--406, 2012.

\bibitem{MST}
L.~Addario-Berry, N.~Broutin, C.~Goldschmidt, and G.~Miermont.
\newblock The scaling limit of the minimum spanning tree of the complete graph.
\newblock {\em Ann. Probab}, 45(5):3075--3144, 2017.

\bibitem{Steal1}
L.~Addario-Berry, S.~Donderwinkel, M.~Maazoun, and J.~Martin.
\newblock A new proof of {C}ayley's formula.
\newblock arxiv:2107.09726, 2021.

\bibitem{AaldousBroder}
D.~Aldous.
\newblock The random walk construction of uniform spanning trees and uniform
  labelled trees.
\newblock {\em Siam J. Discrete Math.}, 3(4):450--465, 1990.

\bibitem{Aldous1}
D.~Aldous.
\newblock The continuum random tree {I}.
\newblock {\em Ann. Probab}, 19:1--28, 1991.

\bibitem{Aldous_exc_ER}
D.~Aldous.
\newblock Brownian excursions, critical random graphs and the multiplicative
  coalescent.
\newblock {\em Ann. Probab}, 25(2):812--854, 1997.

\bibitem{IntroICRT1}
D.~Aldous and J.~Pitman.
\newblock Inhomogeneous continuum random trees and the entrance boundary of the
  additive coalescent.
\newblock {\em Probab. Theory Related Fields}, 118(4):455--482, 2000.

\bibitem{HistoConfigA}
E.~Bender and E.~Canfield.
\newblock The asymptotic number of labeled graphs with given degree sequences.
\newblock {\em J. Combinatorial Theory Ser. A}, 24(3):296--307, 1978.

\bibitem{HomogeneousCase2}
S.~Bhamidi, N.~Broutin, S.~Sen, and X.~Wang.
\newblock Scaling limits of random graph models at criticality: Universality
  and the basin of attraction of the {E}rd\"os-{R}\'enyi random graph.
\newblock arXiv:1411.3417, 2014.

\bibitem{MST2}
S.~Bhamidi, R.~Van~Der Hofstad, and S.~Sen.
\newblock The multiplicative coalescent, inhomogeneous continuum random trees,
  and new universality classes for critical random graphs.
\newblock {\em Probab. Theory Relat. Fields}, 170:387--474, 2018.

\bibitem{Uniform}
A.~Blanc-Renaudie.
\newblock {L}imit of trees with fixed degree sequence.
\newblock arxiv:2110.03378.

\bibitem{ICRT1}
A.~Blanc-Renaudie.
\newblock Compactness and fractal dimension of inhomogeneous continuum random
  trees.
\newblock arxiv:2012.13058, 2020.

\bibitem{HistoConfigB}
B.~Bollob\'{a}s.
\newblock A probabilistic proof of an asymptotic formula for the number of
  labelled regular graphs.
\newblock {\em European J. Combin.}, 1(4):311--316, 1980.

\bibitem{Massart}
S.~Boucheron, G.~Lugosi, and P.~Massart.
\newblock {\em Concentration Inequalities. A Nonasymptotic Theory of
  Independence.}
\newblock Oxford university press, 2013.

\bibitem{HistoMultD}
T.~Britton, M.~Deijfen, and A.~Martin-L\"{o}f.
\newblock Generating simple random graphs with prescribed degree distribution.
\newblock {\em J. Stat. Phys.}, 124(6):1377--1397, 2006.

\bibitem{AldousBbroder}
A.~Broder.
\newblock Generating random spanning trees.
\newblock {\em In Proc. 30'th IEEE Symp. Found. Comp. Sci}, pages 442--447,
  1989.

\bibitem{P-graph-2}
N.~Broutin, T.~Duquesne, and M.~Wang.
\newblock Limits of multiplicative inhomogeneous random graphs and {L}\'evy
  trees: The continuum graphs.
\newblock Ann. Appl. Probab. (to appear) arxiv.org:1804.05871.

\bibitem{P-graph-1}
N.~Broutin, T.~Duquesne, and M.~Wang.
\newblock Limits of multiplicative inhomogeneous random graphs and {L}\'evy
  trees: Limit theorems.
\newblock {\em PTRF}, 2021.

\bibitem{Glue}
D.~Burago, Y.~Burago, and S.~Ivanov.
\newblock {\em A Course in Metric Geometry}, volume~33 of {\em Graduate Studies
  in Mathematics}.
\newblock American Mathematical Society, Providence, RI, 2001.

\bibitem{IntroICRT2}
M.~Camarri and J.~Pitman.
\newblock Limit distributions and random trees derived from the birthday
  problem with unequal probabilities.
\newblock {\em Electron. J. Probab.}, 5(2), 2000.

\bibitem{HistoMultA}
F.~Chung and L.~Lu.
\newblock Connected components in random graphs with given expected degree
  sequences.
\newblock {\em Ann. Comb.}, 6(2):125--145, 2002.

\bibitem{StableConfig}
G.~Conchon-Kerjan and C.~Goldschmidt.
\newblock The stable graph: the metric space of a critical random graph with
  i.i.d power-law degrees.
\newblock arxiv:2002.04954.

\bibitem{Dhara}
S.~Dhara.
\newblock {\em Critical Percolation on Random Networks with Prescribed
  Degrees.}
\newblock PhD thesis, Technische Universiteit Eindhoven, arXiv:1809.03634,
  2018.

\bibitem{HeavyConfig}
S.~Dhara, R.~van~der Hofstad, J.~S.H. van Leeuwaarden, and S.~Sen.
\newblock Heavy-tailed configuration models at criticality.
\newblock {\em Ann. Inst. H. Poincar{\'e} Probab. Statist.}, 56(3):1515 --
  1558, August 2020.

\bibitem{StableMarchal}
C.~Goldschmidt, B.~Haas, and D.~S\'enizergues.
\newblock Stable graphs: distributions and line-breaking construction.
\newblock arxiv.org:1811.06940.

\bibitem{Kallenberg}
O.~Kallenberg.
\newblock {\em Random Measures, Theory and Applications}.
\newblock Springer, 2010.

\bibitem{EquivGP}
W.~L\"ohr.
\newblock Equivalence of gromov-prokhorov and gromov's $\underline
  \square_\lambda$-metric on the space of metric measure spaces.
\newblock {\em Electron. C. Probab.}, 26(1):213--252, 2013.

\bibitem{Newman}
M.~Newman.
\newblock The structure and function of complex networks.
\newblock {\em SIAM review}, 45:167--256, 2003.

\bibitem{HistoMultC}
I.~Norros and H.~Reittu.
\newblock On a conditionally {P}oissonian graph process.
\newblock {\em Adv. in Appl. Probab.}, 38(1):59--75, 2006.

\end{thebibliography}
\bibliographystyle{plain}

\appendix
\section{Appendix}
\subsection{$\R$-tree reconstruction problem} \label{A.1}
Recall that a $\R$-tree is a loopless geodesic metric space. If $\T$ is a $\R$-tree, we say that $x\in \T$ is a leaf of $\T$ if $\T\backslash \{x\}$ is connected. Let $(\T,d)$ be a $\R$-tree with leaves $\{\star_i\}_{1\leq i \leq N}$. In this section we reconstruct a $\R$-tree isometric to $\T$ from $(d_{i,j})_{1\leq i,j \leq N}:=(d(\star_i,\star_j))_{1\leq i,j \leq N}$.
%Let $T$ be a tree with leaves $\{\star_i\}_{1\leq i \leq N}$. Let $d$ be the graph distance on $T$. We explain how to reconstruct a tree isomorph to $(T,d)$ from $(d_{i,j})_{1\leq i,j \leq N}:=(d(\star_i,\star_j))_{1\leq i,j \leq N}$. %We then prove that $\square_c$ is a continuous function $\{d(\star_i,\star_j)\}_{1\leq a,b \leq 2c}$ invariant by scaling. %We finally extend this last result in the continuum setting.

For every $a,b\in \T$ let $\llbracket a,b\rrbracket $ be the geodesic path between $a$ and $b$. Since $\T$ is a $\R$-tree note that for every $1\leq a \neq b\neq c\leq N$ there exists a unique vertex $\star_{a,b,c}$ in $\llbracket \star_a,\star_b \rrbracket \cap \llbracket \star_a,\star_c \rrbracket\cap\llbracket \star_b,\star_c \rrbracket$.% Let  denote this vertex. %The following lemma is the key to reconstruct $\R$-trees.
%First let us introduce some notations. 
%For every $v,w\in \V$ let $\Ee_{v,w}\subset \Ee$ be the set of edges that are on the minimal path between $v$ and $w$. Then, for $1\leq a \neq b\neq c\leq N$, note that since $T$ is a tree there is one and only one vertex in $\Ee_{a,b}\cap \Ee_{a,c}\cap \Ee_{b,c}$. Let $\star_{a,b,c}$ denote this vertex.% Also, for every $1\leq a, b \leq N$, let $d_{a,b}:=d(\star_a,\star_b)$.%We have the following lemma:

\begin{lemma} \label{ZOURF} For every $1\leq a \neq b \neq c\leq N$, $2d(\star_a,\star_{a,b,c})= d_{a,b}+d_{a,c}-d_{b,c}$. \end{lemma}
\begin{proof} Note that $d_{a,b}=d(\star_a,\star_{a,b,c})+d(\star_{a,b,c},\star_b)$, and similarly $d_{b,c}=d(\star_b,\star_{a,b,c})+d(\star_{a,b,c},\star_c)$ and $d_{a,c}=d(\star_a,\star_{a,b,c})+d(\star_{a,b,c},\star_c)$. The desired equality follows by sum. 
% Note that, by a simple counting argument, for every $v\in V\backslash \{\star_{a,b,c}\}$, 
%\begin{equation} 2\1_{v\in \Ee_{\star_a,\star_{a,b,c}}}=\1_{v\in \Ee_{\star_a,\star_b}}+\1_{v\in \Ee_{\star_a,\star_c}}-\1_{v\in \Ee_{\star_a,\star_c}}. \label{14h/5/12} \end{equation}
%The desired result then follows by integrating \eqref{14h/5/12} over all $v\in V\backslash \{\star_{a,b,c}\}$.
\end{proof} To reconstruct $T$ we reconstruct recursively for $1\leq n \leq N$ the subtree spanned by $\{\star_i\}_{1\leq i \leq n}$, which is $\T_n:=\bigcup_{1\leq i,j\leq n} \llbracket \star_i,\star_j\rrbracket $. It is easy to check that for $1\leq n\leq N$, $(\T_n, d)$ is a $\R$-tree. Moreover, note that $\T_{n+1}=\T_n\cup \llbracket W_n, \star_{n+1}\rrbracket $, where $W_n$ is the closest point from $\star_{n+1}$ on $\T_n$. Therefore, it is enough to reconstruct $(W_i)_{1\leq i < N}$ and $(d(W_n,\star_{n+1}))_{1\leq n <N}$. This suggest the following construction. Below, $(\beta_i)_{i\in \N}$ is the canonical base of $\R^{+\N }$.
%For every $1\leq n\leq N$ let $T_n$ be the subtree of $T$ spanned by $\{\star_i\}_{1\leq i \leq n}$. For every $1\leq n \leq N+1$, to reconstruct $T_{n+1}$ it is enough to reconstruct $T_n$, the projection of $\star_{n+1}$ on $T_n$, and compute the distance between $\star_{n+1}$ and this projection. Furthermore, Lemma \ref{ZOURF} describe precisely the projection $\star_{n+1,a,b}$ of $\star_{n+1}$ on $\Ee_{a,b}$ for $a<n\leq n$. By taking the closer projection we recover the global projection and thus we have the following construction: Let $(V_{i,j})_{i,j\in \N}$ be a set of vertices distinct from $\{\star_i\}_{i\in \N}$.

\begin{algorithm} \label{reconstruct} Reconstruction of a $\R$-tree on $(\R^{+\N},\| \|_{\infty})$  from $M=(d_{i,j})_{1\leq i,j \leq N}$.
\begin{compactitem}
\item[-] Let $\star^M_1:=0$. Let $\T^M_1=(\{\star^M_1\},\emptyset)$. 
\item[-] Let $\star^M_2:=(0,d(\star^M_1,\star^M_2))$ Let $\T^M_2:=\{x\beta_1, 0\leq x\leq d(\star^M_1,\star^M_2)\}$.
\item[-] For every $2\leq n <N$: 
\begin{compactitem}
\item[-] Let $1\leq b^M_{n}\neq c^M_{n}\leq n$ be the smallest integers (for some predetermined order) that minimize $d_{n+1,b^M_n}+d_{n+1,c^M_n}-d_{b^M_n,c^M_n}$. 
\item[-] Let $W^M_{n}$ be the vertex of $\T^M_n$ at distance $d_{b^M_n,n+1}+d_{b^M_n,c^M_n}-d_{n+1,c^M_n}$ of $\star^M_{b_n}$ and at distance $d_{c^M_n,n+1}+d_{c^M_n,b^M_n}-d_{n+1,b^M_n}$ of $\star^M_{c^M_n}$. (See below for existence and unicity.)
\item[-] Let $\T^M_{n+1}:= \T^M_n\cup \{W^M_n+x\beta_{n},0\leq x\leq d_{n+1,b^M_n}+d_{n+1,c^M_n}-d_{b^M_n,c^M_n}\}$.
\end{compactitem}
\item[-] Let $\T^M:=\T^M_N$. 
\end{compactitem}
\end{algorithm}
\begin{remark} The idea of constructing subtrees on $(\R^{+\N},d_{\infty})$ comes from Aldous \cite{Aldous1}.
\end{remark}
\begin{lemma} \label{Lemme de reconstruction} Let $(\T,d)$ be a $\R$-tree with leaves $\{\star_i\}_{1\leq i \leq N}$. Let $M=(d(\star_i,\star_j))_{1\leq i,j \leq N}$. Then: 
\begin{compactitem}
\item[a)] For every $1\leq n\leq N$, $\T^M_n$ is well defined.
\item[b)] $(\T,d,\star_1,\dots, \star_N)$ and $(\T^M,d_{\infty},\star^M_1,\dots, \star^M_N)$ are isometric (see Section \ref{PointedGH}).
\end{compactitem}
\end{lemma}
%\begin{lemma} Let $T=(\V,\Ee)$ be a tree with leaves $\{\star_i\}_{1\leq i \leq N}$. Let $M=(d(\star_i,\star_j))_{1\leq i,j \leq N}$. Let $d^M$ be the graph distance on $T^M$. Then $(\V,d)$ and $(\V^M,d^M)$ are isomorph. 
%\end{lemma}
\begin{proof} We prove by induction that for $1\leq n\leq N$, $\T_n^M$ is well defined and that  $(\T_n,d,\star_1,\dots, \star_n)$ and $(\T_n^M,d_{\infty},\star^M_1,\dots, \star^M_n)$ are isometric. First if $n=1$ or $n=2$ then the result is obvious. Then let $2\leq n<N$ such that  $\T_n^M$ is well defined and such that there exists an isometry $\phi_n$ from $(\T_n,d,\star_1,\dots, \star_n)$ to $(\T_n^M,d_{\infty},\star^M_1,\dots, \star^M_n)$.

Recall that $\T_n=\bigcup_{1\leq i,j\leq n} \llbracket \star_i,\star_j\rrbracket$, and that $W_n$ is the closest point from $\star_{n+1}$ on $\T_n$. So there exist $1\leq b_n\neq c_n\leq n$ such that $W_n\in  \llbracket \star_{b_n},\star_{c_n} \rrbracket$. Hence $W_n=\star_{n+1,b_n,c_n}$. 
Then by Lemma \ref{ZOURF},
\begin{equation} d(\star_{n+1},W_n)=d(\star_{n+1},\star_{n+1,b_n,c_n})=d_{n+1,b_n}+d_{n+1,c_n}-d_{b_n,c_n}. \label{BADABADABADA} \end{equation}
Also, by Lemma \ref{ZOURF}, since $W_n$ is the closest point from $\star_{n+1}$ on $\T_n$,
\[ d(\star_{n+1},W_n)\leq \min_{1\leq b\leq c\leq n}d(\star_{n+1},\star_{n+1,b,c})= \min_{1\leq b\leq c\leq n} d_{n+1,b}+d_{n+1,c}-d_{b,c}. \]
Therefore, we may assume that $b_n=b^M_n$ and that $c_n=c^M_n$.

Furthermore, since $W_n=\star_{n+1,b_n,c_n}$, 
\begin{equation} d(W_n,\star_{b_n})=d_{b_n,n+1}+d_{b_n,c_n}-d_{n+1,c_n}\quad ; \quad d(W_n,\star_{c_n}) = d_{c_n,n+1}+d_{c_n,b_n}-d_{n+1,b_n}.\label{21h/5/12} \end{equation}
$W_n$ is the only vertex of $\T_n$ satisfying \eqref{21h/5/12}. Indeed, any vertex $V$ satisfying \eqref{21h/5/12} must also satisfy 
\[ d(V,\star_{b_n})+d(V,\star_{c_n})=d(\star_{b_n},\star_{c_n}),\] 
and so must be $W_n$, the only vertex of $\llbracket \star_{b_n},\star_{c_n} \rrbracket$ at distance $d_{b_n,n+1}+d_{b_n,c_n}-d_{n+1,c_n}$ of $\star_{b_n}$. 

Then, by definition of $\phi_n$ and \eqref{21h/5/12}, $\phi_n(W_n)$ is the only vertex of $T^m_n$ satisfying
\begin{equation*} d_\infty(\phi_n(W_n),\star^M_{b_n})=d_{b_n,n+1}+d_{b_n,c_n}-d_{n+1,c_n}\quad ; \quad d_\infty(\phi_n(W_n),\star^M_{b_n}) = d_{c_n,n+1}+d_{c_n,b_n}-d_{n+1,b_n}. \end{equation*}
Therefore, $W_n^M$ and thus $\T_n^{M+1}$ are well defined.

Finally recall that $\T_{n+1}=\T_n\cup \llbracket W_n,\star_{n+1}\rrbracket$. Then by definition of $\T^M_{n+1}$, \eqref{BADABADABADA} and $\phi_n(W_n)=W^M_n$, we have $\T^M_{n+1}= \T^M_n\cup \{\phi_n(W_n)+x\beta_{n}, 0\leq x\leq d(W_n,\star_{n+1})\}$. Also both union are disjoint, so one can extend $\phi_n$ to an isometry $\phi_{n+1}$ from $\T_{n+1}$ to $\T^M_{n+1}$ such that for every $x\in \llbracket W_n,\star_{n+1}\rrbracket$, $\phi_{n+1}(x):=\phi_n(W_n)+d(W_n,x)\beta_n$. This concludes the proof.
\end{proof}
We now prove a corollary, which we use to prove Theorem \ref{THM2}.
\begin{lemma} \label{reconstructTHM} Let $((\T_n,d_n))_{n\in \N}$ be a sequence of $\R$-trees with leaves $\{\star^n_i\}_{1\leq i \leq N}$. Assume that 
\[ \forall 1\leq i\neq j\leq n,\quad d_n(\star^n_i,\star^n_j)\limit d_{i,j}\in \R^{+*} .\]
Then there exist a unique $N$-pointed $\R$-tree $(\T,d,(\star_1,\dots, \star_N))$ up to isometry such that for every $1\leq i,j \leq N$, $d(\star_i,\star_j)=d_{i,j}$. Moreover, $((\T_n,d_n,(\star^n_i)_{1\leq i \leq N}))_{n\in \N}$ converges for the $N$-pointed Gromov--Hausdorff topology (see Section \ref{PointedGH}) toward $(\T,d,(\star_i)_{1\leq i \leq N})$.
\end{lemma}
\begin{proof} First uniqueness follows from Lemma \ref{Lemme de reconstruction}. Let us prove existence.
For every $n\in \N$ let $M^n:=(d^n(\star^n_i,\star^n_j))_{1\leq i,j \leq N}$. Similarly let $M=(d_{i,j})_{1\leq i,j \leq N}$. Note that for every $n\in \N$, %$\T^{M_n}$ is a compact subspace of $(\R^{+\N},\| \|_{\infty})$ and that,
\begin{equation} \T^{M_n}\subset \left \{x\in \R^{+N}\times \{0\}^\N, d_\infty(0,x)\leq \max_{1\leq j\leq N} d^n(\star^n_1,\star^n_i) \right \}. \label{1h/6/12}\end{equation}
Thus for every $1\leq m < N$, $\{W^{M_n}_m \}_{n\in \N}$ is tight.

Let $(n_i)_{i\in \N}$ be an increasing sequence of integer such that for every $1\leq m < N$, $(W^{M_{n_i}}_m )_{i\in \N}$ converges toward $W^\infty_m$. Then, intuitively, the whole Algorithm \ref{reconstruct} converges. More precisely, $\T^{M_{n_i}}$ converges for the Hausdorff distance toward a $\R$-tree $\T$ that is constructed from Algorithm \ref{reconstruct} with entry $M$ and where for $1\leq m<N$, $W_m^M$, %whose definiteness is not clear, 
is replaced by $W^\infty_m$.  Furthermore, for every $1\leq m\leq N$, $(\star^{M_{n_i}}_m )_{i\in \N}$ converges toward $\star^\infty_n$ which is also obtained from the same algorithm.% \ref{reconstruct} with the same modifications. %(This ensure the definiteness of the algorithm.)

Then it is easy to check that the leaves of $\T$ are $(\star^\infty_m )_{1\leq m\leq N}$, and that for every $1\leq i,j\leq N$,
\[ d(\star_i^\infty, \star^\infty_j)=\lim_{m\to \infty} d(\star^{n_m}_i,\star_j^{n_m})=\lim_{m\to \infty } d_{n_m}(\star^n_i,\star^n_j)=d_{i,j}. \]
Therefore $(\T,d)$ satisfies the properties described in the lemma. %(This tree can be obtained from $(\T,d_\infty)$ by, roughly speaking, replacing $(\star^\infty_m )_{1\leq m\leq N}$ by $(\star_m )_{1\leq m\leq N}$.)

Finally, let us prove the convergence. First, the right-hand side of \eqref{1h/6/12} is compact so $(\T^{M_n})_{n\in \N}$ is a tight sequence for the Hausdorff topology. Then from any converging subsequence of $(\T^{M_n})_{n\in \N}$ we may further extract such that $(W^{M_{n_i}}_m )_{i\in \N}$ converges. It then follows from the first part of the proof that $(T^{M_n} )_{n\in \N}$ converges for the Hausdorff distance toward $\T$. Finally by Lemma \ref{Lemme de reconstruction} for every $n\in \N$, $(\T^{M_n},d_\infty,(\star^M_i)_{1\leq i \leq N})$ and $(\T_n,d_n,(\star^n_i)_{1\leq i \leq N})$ are isometric. The desired convergence follows.
%The right hand side is compact so $\T^{M_n}$ is a tight sequence for the Hausdorff topology.
%
%Let $(n_i)_{i\in \N}$ be an increasing sequence of integer such that $\T^{M_{n_i}}$ converges for the Hausdorff distance. Let $X$ be its limit. First $(X,d_\infty)$ is a $\R$-tree as a limit of $\R$-tree (the four-point condition stay as the limit, see e.g. \cite{GHP} Proposition 3.2).
%
%Moreover, we may assume up to further extraction that for every $1\leq i \leq N$, $W_i^{M_{n_m}}\to W^\infty_i\in \R^{+\N}$ as $m\to \infty$. Note that for every $1\leq i \leq N$, $\star^\infty_i\in X$, and that for every $1\leq i,j\leq N$,
%\[ d(\star_i^\infty, \star^\infty_j)=\lim_{m\to \infty} d(\star^{n_m}_i,\star_j^{n_m})=\lim_{m\to \infty } d_{n_m}(\star_i,\star_j)=d_{i,j}. \]
%Therefore to prove existence note that it is enough to show that the leaves of $X$ are $\{\star^\infty_i\}_{1\leq i \leq N}$.
%
%We may further assume that for every $1\leq i \leq N$n $W^{M_{n_i}}$ converges toward $W^\infty \in $. It directly follows that the whole construction of $T^M_{n}$, in other words if $X'$ denote the $\R$-tree constructed by replacing in Algorithm \ref{reconstruction}, $W^M_i$ by $W^\infty_i$ then $T^M_{n}$ converges toward $X$
\end{proof}
\subsection{$\square_c$ is a continuous function of the matrix distance}
Recall Section \ref{TkDef}. Let us extend $\square_i$ to general $\R$-trees. Note that for every $\R$-tree $(\T,d)$, one may define a Borel measure $\lambda$ on $\T$ such that for every $a,b\in \T$, %writing $\llbracket a, b\rrbracket$ for the geodesic path between $a$ and $b$ in $\T$, 
$\lambda \llbracket a, b\rrbracket =d(a,b)$. By analogy with $\R$ we call $\lambda$ the Lebesgue measure. For $c\in \N$, if $\{\star_i\}_{1\leq i \leq 2c}$ are leaves of $\T$, we let $\cyc_c(\T)$ be the set of all $x\in \R$ such that $\GT_{(\star_i)_{1\leq i \leq 2c}}(\T)\backslash\{x\}$ is connected. By Lemma \ref{YOLLOOO} below $\cyc_c(\T)$ is measurable. Let $\square_c(\T)$ be its Lebesgue measure.

 It is easy to check that this definition of $\square_c$ extends the definition of $\square_c$ described in Section \ref{TkDef} and informally equals $\square_c+c$ where $\square_c$ is defined in the discrete setting in Section \ref{DkDef}. The goal of this section is to prove a continuity result for $\square_c$.
\begin{lemma} \label{YOLLOOO} For every $c\in \N$, for every $\R$-tree $(\T,d)$, if $\{\star_i\}_{1\leq i \leq 2c}$ are  leaves of $\T$ then 
\[ {\cyc}_c(\T)=\bigcup_{1\leq b\leq c} \llbracket \star_{2b-1},\star_{2b}\rrbracket . \]
\end{lemma} 
\begin{proof} On the one hand, for every $b\leq c$, $\llbracket \star_{2b-1},\star_{2b}\rrbracket \subset \cyc(\G_{(\star_{i})_{1\leq i \leq 2c}}(T))$ since $\llbracket \star_{2b-1},\star_{2b}\rrbracket$ is a cycle in $\GT_{(y_i)_{1\leq i \leq 2c}}(\T)$ (a geodesic path that have the same starting and ending point).

On the other hand, let $x\in \T \backslash \bigcup_{1\leq b\leq c} \llbracket \star_{2b-1},\star_{2b}\rrbracket$. If $\T\backslash \{x\}$ is connected then $x\notin {\cyc}_c(T)$ since $\GT_{(y_i)_{1\leq i \leq 2c}}(\T)\backslash\{x\}$ is also connected. Otherwise $\T\backslash \{x\}$ is disconnected. Let $\T_1,\T_2$ be the two connected components of $\T\backslash \{x\}$. For every $1\leq b \leq 2c$ note that since $x\notin \llbracket \star_{2b-1},\star_{2b}\rrbracket $, either $\star_{2b-1},\star_{2b}\in \T_1$ or $\star_{2b-1},\star_{2b}\in \T_2$. Therefore, by induction, for every $1\leq b \leq c$, $\GT_{(y_i)_{1\leq i \leq 2b}}(\T)\backslash\{x\}$ is still disconnected. In other words, $ x\notin {\cyc}_c(T)$.
\end{proof}
\begin{lemma} \label{CHIANTa} Let $c\in \N$. There exists a continuous function $f_c:\R^{2c\times 2c}$, such that for every $\R$-tree $(\T,d)$ such that $\{\star_i\}_{1\leq i \leq 2c}$ are leaves of $\T$,
\[ \square_c((\T,d))=f_c((d(\star_i,\star_j))_{1\leq i,j \leq 2c}). \]
Furthermore for every $\lambda\in \R^+$, $\square_c((\T,\lambda d))=\lambda \square_c((\T,d))$.
\end{lemma}
\begin{proof} Fix $c\in \N$. note that $\square_c$ is invariant under isometry so Lemma \ref{Lemme de reconstruction} imply that $f_c$ exists.   Also, the scaling property is straightforward from the initial definition since rescaling $d$ rescale the Lebesgue measure. Thus, it remains to prove the continuity property.

To this end, we prove an explicit formula for $\square_c((\T,d))-\square_{c-1}((\T,d))$ using Lemma \ref{YOLLOOO}. %Since $\square_0=0$ is clearly continuous this will allow us to prove by induction that $f_c$ is continuous. 
Let $M:=(d_{i,j})_{1\leq i ,j\leq 2c}:= (d(\star_i,\star_j))_{1\leq i,j \leq 2c}$. Since $(\T,d)$ is a $\R$-tree we may define $\varphi_{c}$ as the unique isometry from $[0,d_{{2i-1},2i}]$ to $\llbracket \star_{2i-1}, \star_{2i}\rrbracket$ such that $\varphi_{c}(0)=\star_{2i-1}$ and $\varphi_{c}(d_{{2i-1},2i})=\star_{2i}$. We have by the transport formula,
\begin{align} \square_c((\T,d))-\square_{c-1}((\T,d)) & = \lambda \bigg (\llbracket \star_{2i-1}, \star_{2i}\rrbracket \backslash \bigcup_{1\leq b<c} \Ee_{\llbracket \star_{2b-1},\star_{2b} \rrbracket} \bigg ) \notag
\\ & =\int_0^{d_{2i-1,2i}} \prod_{b=1}^{c-1} \1_{\varphi_c(x)\notin {\llbracket \star_{2b-1},\star_{2b} \rrbracket}} dx. \label{TapaTapaTapapapapa}
\end{align}

Then, since $\T$ is a $\R$-tree, for every $1\leq b <c$, $\llbracket \star_{2b-1},\star_{2b} \rrbracket\cap \llbracket \star_{2c-1},\star_{2c} \rrbracket$ is a  segment. For every $1\leq b <c$, let $I_b$ be the real interval such that $x\in  I_b$ iff $\varphi_c(x)\in \llbracket \star_{2b-1},\star_{2b} \rrbracket$. Intuitively, by \eqref{TapaTapaTapapapapa} it is enough to show that for $1\leq b<c$, $I_b$ may be seen as a continuous function of $M$. Indeed, this would directly imply that $f_c(M)-f_{c-1}((d_{i,j})_{1\leq i ,j\leq 2c-2})$ is continuous. And the desired result would then follow by induction.

Thus let us fix $1\leq b <c$, and let us compute $I_b$. For every $a,b,c\in \T$,  $\star_{a,b,c}$ let be the unique vertex in $\llbracket \star_a,\star_b \rrbracket \cap \llbracket \star_a,\star_c \rrbracket\cap\llbracket \star_b,\star_c \rrbracket$. Since $\T$ is a $\R$-tree, note that 
\[ d(\star_{2c-1},\star_{2c-1,2c, 2b-1})\neq d(\star_{2c-1},\star_{2c-1,2c, 2b}) \Longrightarrow I_b= [d(\star_{2c-1},\star_{2c-1,2c, 2b-1}),d(\star_{2c-1},\star_{2c-1,2c, 2b})]^+, \]
where for $x,y\in \R$, $[x,y]^+:=[\min(x,y),\max(x,y)]$. Also note that
\[ d(\star_{2c-1},\star_{2c-1,2c, 2b-1})\neq d(\star_{2c-1},\star_{2c-1,2c, 2b})\Longrightarrow I_b\in \{\emptyset,\{d(\star_{2c-1},\star_{2c-1,2c, 2b})\}\}.\]
 Moreover, by Lemma \ref{ZOURF},
\[d(\star_{2c-1},\star_{2c-1,2c, 2b-1})=d_{2c-1,2c}+d_{2c-1,2b-1}-d_{2c,2b-1}, \]
and 
\[d(\star_{2c-1},\star_{2c-1,2c, 2b})=d_{2c-1,2c}+d_{2c-1,2b}-d_{2c,2b}. \]
Therefore $I_b$ may be seen as a continuous function of $M$. Finally \eqref{TapaTapaTapapapapa} concludes the proof.
\end{proof}
\end{document}